\documentclass[review]{elsarticle}

\usepackage{amsmath}
\usepackage{amssymb}
\usepackage{amsthm}
\usepackage{amsfonts}
\usepackage{graphicx}
\usepackage{algorithm}
\usepackage{algorithmic}
\usepackage{verbatim}
\usepackage{caption}
\usepackage{subcaption}
\usepackage{float}
\usepackage{booktabs}
\usepackage{array}
\usepackage{lineno,hyperref}
\usepackage[nameinlink,noabbrev]{cleveref}
\usepackage{amsopn}
\usepackage{calc}
\usepackage{bbm}
\usepackage{multirow}
\usepackage{tabularray}
\usepackage{tikz}
\usepackage{ragged2e}
\numberwithin{equation}{section}
\DeclareMathOperator*{\esssup}{ess\,sup}

\newcommand {\average}[1] {\mbox{$\left\{\!\!\left\{ #1 \right\}\!\!\right\}$}}
\newcommand {\jump}[1] {\mbox{$\left[\!\left[ #1 \right]\!\right]$}}
\newcommand{\set}[2]{\left\{{#1}\,:~{#2}\right\}}

\newtheorem{theorem}{Theorem}[section]

\newtheorem{remark}[theorem]{Remark}

\newtheorem{assumption}{Assumption}

\algsetup{linenosize=\tiny}
\makeatletter

\newcommand{\DESCRIPTION@original@item}{}
\let\DESCRIPTION@original@item\item
\newcommand*{\DESCRIPTION@envir}{DESCRIPTION}
\newlength{\DESCRIPTION@totalleftmargin}
\newlength{\DESCRIPTION@linewidth}
\newcommand{\DESCRIPTION@makelabel}[1]{\llap{#1}}%
\newcommand{\DESCRIPTION@item}[1][]{%
	\setlength{\@totalleftmargin}%
	{\DESCRIPTION@totalleftmargin+\widthof{\textbf{#1 }}-\leftmargin}%
	\setlength{\linewidth}
	{\DESCRIPTION@linewidth-\widthof{\textbf{#1 }}+\leftmargin}%
	\par\parshape \@ne \@totalleftmargin \linewidth
	\DESCRIPTION@original@item[\textbf{#1}]%
}

\modulolinenumbers[5]

\journal{Journal of \LaTeX\ Templates}
\begin{document}

\begin{frontmatter}



\title{An Adaptive Algorithm Based on Stochastic Discontinuous Galerkin for Convection Dominated Equations with Random Data}

\author[TUchem]{Pelin \c{C}\.{i}lo\u{g}lu}
\ead{pelin.ciloglu@mathematik.tu-chemnitz.de}

\author[iam]{Hamdullah Y\"ucel\corref{cor1}}
\ead{yucelh@metu.edu.tr}

\cortext[cor1]{Corresponding author}

\address[TUchem]{Department of Mathematics, Technische Universit\"{a}t Chemnitz, 09126, Chemnitz, Germany}
\address[iam]{Institute of Applied Mathematics, Middle East Technical University, 06800 Ankara, Turkey}

\begin{abstract}
In this paper, we propose an adaptive approach, based on mesh refinement or parametric enrichment with polynomial degree adaption, for numerical solution of convection dominated equations with random input data. A parametric system emerged from an application of stochastic Galerkin approach is discretized by using  symmetric interior penalty Galerkin (SIPG) method with upwinding for the convection term in the spatial domain. We derive a residual-based error estimator contributed by the error due to the SIPG discretization,  the (generalized) polynomial chaos discretization in the stochastic space, and data oscillations. Then, the reliability of the proposed error estimator, an upper bound for the energy error up to a multiplicative constant, is shown. Moreover, to balance the errors stemmed from spatial and stochastic spaces, the truncation error emerged from Karhunen--Lo\`{e}ve expansion  are  considered in the numerical simulations. Last, several benchmark  examples including a random diffusivity parameter, a random convectivity parameter, random diffusivity/convectivity parameters, and a random (jump) discontinuous diffusivity parameter, are tested to illustrate the performance of the proposed estimator.        
\end{abstract}




\end{frontmatter}



\section{Introduction} \label{sec:intro}

With the improvement of computer--processing capacities, the demand for efficient numerical simulations of partial differential equations (PDEs) with uncertainty or parameter--dependent inputs, which are widely used  in many fields in science and engineering, see, e.g., \cite{JPDelhomme_1979,SMarkov_BCheng_AAsenov_2012,DZhang_QKang_2004}, has begun to grow. Uncertainty generally arises from the lack of knowledge (called as epistemic uncertainty) or the nature variability in the system (called as aleatoric uncertainty). While modelling the system, uncertainty can be characterized by the random fields in the coefficients, the boundary conditions, the initial conditions, or the geometry of the underlying PDE. 

While solving a forward uncertainty quantification problem, the aim is, in general, to determine the effect of uncertainties in the input on the solution of the underlying problem or to investigate the numerical behaviour of the statistical moments of the solution, such as mean and variance. In the last decade, there exist a vast of studies on the elliptic  equations with random data, for instance, based on Monte Carlo (MC) and its variants, such as multilevel or quasi MC methods \cite{ABarth_CSchwab_NZollinger_2011,KACliffe_MBGiles_RScheichl_ALTeckentrup_2011,IGGraham_FYKuo_DNuyens_RScheichl_IHSloan_2011}, stochastic collocation (SC) \cite{IBabuska_FNobile_RTempone_2007a,FNobile_RTempone_CGWebster_2008a}, perturbation methods \cite{IBabuska_PChatzipantelidis_2002a,MKleiber_TDHien_1992}, and stochastic Galerkin (SG) \cite{IBabuska_RTempone_GEZouraris_2004a,CSchwab_CJGittelson_2011a}. The stochastic Galerkin method, in contrast to the MC and SC methods, is a nonsampling method and is based on the orthogonality of the residual to the polynomial space as done in the deterministic setting. In addition,  the SG method allows the spatial space to be separated from the stochastic space. We can therefore adapt existing well-known numerical techniques such as a posteriori error analysis or adaptive refinement  in the deterministic setting. Hence, with the help of judiciously chosen adaptive approaches in both spatial and (stochastic) parametric domains, one can avoid a fast growth of the dimension of tensor basis consisting of discontinuous finite element basis functions in the spatial domain and (generalized) polynomial chaos polynomials in the stochastic space.

Design and theoretical analysis of adaptive finite element methods  pioneered by the work of Babu\v{s}ka and Rheinbold  \cite{IBabuska_WCRheinboldt_1978a} have become a popular approach for the efficient solution  of deterministic PDEs \cite{MAinsworth_JTOden_2000a,RVerfurth_1996} as well as PDEs with random data. Within the SG method setting for elliptic PDEs  containing uncertainty, several adaptive strategies based on, for instance, implicit error estimators \cite{XWan_GEKarniadakis_2009}, goal--oriented  a posteriori error estimates \cite{ABespalov_DPraetorius_LRocchi_MGuggeri_2019,LMathelin_OLeMaitre_2007}, multilevel goal--oriented adaptive approaches  \cite{ABespalov_DPraetorius_MGuggeri_2022}, local equilibrium error estimates \cite{MEigel_CMerdon_2016}, hierarchal error estimates \cite{ABespalov_LRocchi_2018,ABespalov_DSilvester_2016,ABespalov_DPraetorius_LRocchi_MGuggeri_2019b,ABespalov_DPraetorius_LRocchi_MGuggeri_2019,AJCrowder_CEPowell_ABespalov_2019}, and residual--based error estimators \cite{MEigel_CJGittelson_CSchwab_EZander_2014,MEigel_CJGittelson_CSchwab_EZander_2015,MEigel_MMarschall_MPfeffer_RSchneider_2020,CJGittelson_2011}, are used to enhance the computed solution and drive the convergence of approximations. Unlike the aforementioned works, here we will focus on more challengeable model problem, namely, convection diffusion equations (especially convection dominated PDEs) containing random coefficients, and to the best of our knowledge, there is no work on the convection dominated equations containing random data in the context of adaptive methods.


In this paper, we mainly aim to construct an adaptive loop,  which consists of successive loops of the following sequence:
\begin{align}\label{loop}
	\textbf{SOLVE   } \rightarrow \textbf{   ESTIMATE   } \rightarrow \textbf{   MARK   } \rightarrow \textbf{   REFINE}
\end{align}
for convection diffusion equations containing random coefficients.  In  addition to the randomness in the input data, such kinds of PDEs also exhibit boundary/interior layers, localized regions where the derivative of the solution changes extremely rapidly. Therefore, we need here an efficient adaptive mechanism, which produce an accurate estimation of the error, while not increasing computational effort. The \textbf{SOLVE} step in the loop~\eqref{loop} stands for the numerical approximation of the statistical moments obtained by the numerical schemes in the current mesh and index set. Here, we use symmetric interior penalty Galerkin (SIPG) method  with upwinding for the convection term to discretize the parametric system of deterministic convection diffusion equations obtained by an application of the stochastic Galerkin approach. In the \textbf{ESTIMATE} step,  error indicators are computed in terms of the discrete solution without knowledge of the exact solutions. They are crucial in designing algorithms to perform a balanced relation between the spatial  and  stochastic components of Galerkin approximations. Obtained total error estimator is contributed by the error due to the (generalized)  polynomial chaos discretization, the error due to the SIPG discretization of the (generalized)  polynomial chaos coefficients in the expansion, and the error due to the data oscillations. Moreover, adding the truncation error emerged from Karhunen--Lo\`{e}ve expansion,  we construct a good balance between the errors emerged from the spatial and stochastic  domains. Based on the information of the estimators in the \textbf{ESTIMATE} step, the \textbf{MARK} step selects a subset of elements subject to refinement and of parametric indices for enrichment. Last, the \textbf{REFINE} step either performs a local refinement of the current triangulation in the spatial domain or enriches the current index set of the parametric domain based on the dominant error estimator contributing to the total estimator. 

This paper is structured as follows. In Section~\ref{sec:prob}, the formulation of the model problem  is presented with the appropriate notation. Section~\ref{sec:sgd} provides an overview of the numerical schemes consisting of Karhunen--Lo\`{e}ve (KL) expansion, stochastic Galerkin method with the symmetric interior penalty Galerkin (SIPG) method for the spatial discretization. A reliable error estimator up to a multiplicative constant in the energy norm  is obtained in Section~\ref{sec:estimator}.  Adaptive loop and numerical results are presented in Section~\ref{sec:num}, and  we conclude with final remarks in Section~\ref{sec:conc}.



\section{Problem formulation} \label{sec:prob}

Let $\mathcal{D}$ be a bounded Lipschitz domain in $\mathbb{R}^2$ with polygonal boundary $\partial \mathcal{D}$, and  triplet $(\Omega, \mathcal{F}, \mathbb{P})$ consisting 
of the set of outcomes  $\Omega$, $\sigma$--algebra of events $\mathcal{F} \subset 2^{\Omega}$, and probability measure $\mathbb{P}: \mathcal{F} \rightarrow [0,1]$ be a complete 
probability space. A real--value square integrable random variable on the probability space $(\Omega, \mathcal{F}, \mathbb{P})$ for a fixed $\boldsymbol{x}\in \mathcal{D}$ represented by
\begin{equation*}
	z(\boldsymbol{x}, \cdot) \in L^2(\Omega) : = \{ z: \Omega \rightarrow \mathbb{R}\, : \; \int_{\Omega} |z(\omega)|^2 \, d\mathbb{P}(\omega) < \infty  \}
\end{equation*}
contains some statistical information such as expected value (i.e., mean) and covariance, respectively, defined by
\begin{subequations}\label{eqn:cov}
	\begin{eqnarray}
		\overline{z}(\boldsymbol{x})&:=& \int_{\Omega} z(\boldsymbol{x},\cdot) \, d\mathbb{P}(\omega) \qquad \quad \; \boldsymbol{x} \in \mathcal{D}, \\
		\mathcal{C}_{z}(\boldsymbol{x},\mathbf{y}) &:=& \int_{\Omega} (z(\boldsymbol{x},\cdot) - \mathbb{E}[ z ]) (z(\boldsymbol{y},\cdot) - \mathbb{E}[ z ] ) \, 
		d\mathbb{P}(\omega) \quad \;\;  \boldsymbol{x},\boldsymbol{y} \in \mathcal{D}. \label{eqn:covb}
	\end{eqnarray}
\end{subequations}
Variance and standard derivation of $z$ are also given by  $\mathcal{V}_{z} = \mathcal{C}_{z}(\boldsymbol{x},\boldsymbol{x})$ and 
$\kappa_{z}= \sqrt{\mathcal{V}_{z}}$, respectively. Moreover, for a given separable Hilbert space $H$ equipped with the norm $\| \cdot \|_H$ 
and seminorm $| \cdot |_H$,  we  introduce a Bochner--type space $L^p(\Omega; H)$ for a random variable $z: \Omega \rightarrow H$ as follows
\begin{equation*}
	L^p(\Omega;H): = \{z: \Omega \rightarrow H :  z \; \; \text{strongly measurable}, \;  \|z\|_{L^p(\Omega;H)} < \infty  \},
\end{equation*}
where
\[
\|z(\omega)\|_{L^p(\Omega;H)}  = \left\{
\begin{array}{ll}
	\left( \int_{\Omega} \|z(\omega)\|^p_H \, d\mathbb{P}(\omega) \right)^{1/p}, & \hbox{for  }  1 \leq p < \infty,\\
	\esssup \limits_{\omega \in \Omega} \|z(\omega)\|_H, & \hbox{for  } p = \infty.
\end{array}
\right.
\]
Further, the following  isomorphism relation 
\[
H(\mathcal{D}) \otimes L^p(\Omega) \simeq L^p(\Omega;H(\mathcal{D})) \simeq  H(\mathcal{D};L^p(\Omega))
\]
holds; see \cite{IBabuska_RTempone_GEZouraris_2004a} for more details.

In this paper, we study a convection diffusion equation containing randomness in the coefficients: Find 
$u: \overline{\mathcal{D}} \times \Omega \rightarrow \mathbb{R}$ such that the following equation holds for almost every $ \omega \in \Omega$
\begin{subequations}\label{eqn:m1}
	\begin{align}
		-\nabla \cdot ( a(\boldsymbol{x},\omega) \nabla u(\boldsymbol{x},\omega) ) + \mathbf{b}(\boldsymbol{x},\omega)\cdot \nabla u(\boldsymbol{x},\omega)  
		& =  f(\boldsymbol{x})  \quad  \hbox{   in} \;\; \mathcal{D} \times \Omega, \\
		u(\boldsymbol{x},\omega) & =  u^d(\boldsymbol{x})  \quad   \hbox{on} \;\; \partial \mathcal{D} \times \Omega,
	\end{align}
\end{subequations}
where  $a:(\mathcal{D} \times \Omega) \rightarrow \mathbb{R}$ and $\mathbf{b}:(\mathcal{D} \times \Omega) \rightarrow \mathbb{R}^2$ are  random diffusivity and convectivity 
coefficients. Source function and Dirichlet boundary condition denoted by  $f$ and $u^d$, 
respectively, are given in a  deterministic way. The well--posedness of the model problem \eqref{eqn:m1} can be shown by proceeding the classical Lax--Milgram 
Lemma, see, e.g., \cite{IBabuska_RTempone_GEZouraris_2004a}, under the following assumption.
\begin{assumption}\label{assump_random}
	For almost every  $\boldsymbol{x} \in  \mathcal{D}$, the diffusion coefficient satisfies
	\[ 
	0 < a_{min} \leq  a(\boldsymbol{x},\omega) \leq a_{max} < \infty,
	\] 
	where $a_{min} = \inf \limits_{\boldsymbol{x} \in \mathcal{D}} a(\boldsymbol{x},\omega)$ and $a_{max} = \sup \limits_{\boldsymbol{x} \in \mathcal{D}} a(\boldsymbol{x},\omega)$ 
	for all $\omega \in  \Omega$, with  $a(\cdot,\omega) \in C^{1}(\Omega)$. Moreover, for almost every $\omega \in \Omega$,  the convection term 
	$\mathbf{b}(\cdot, \omega) \in \big( L^{\infty}(\overline{\mathcal{D}}) \big)^2$  is incompressible, that is, $\nabla \cdot \mathbf{b}(\boldsymbol{x},\omega) =0$.
\end{assumption}

\begin{remark}
	In this study, only coefficients are assumed to be random  variables but it would not introduce extra difficulties in terms of the  implementation of proposed method to also consider the source function $f$ and Dirichlet boundary condition $u^d$ as random fields not correlated to the coefficients as long as appropriate integrability of the data holds. In addition, the underlying random coefficients can be modelled by a continuous part and a discontinuous part (e.g., jump term) as discussed in \cite{ABarth_AStein_2018,ABarth_AStein_2022,JLi_XWang_KZhang_2016}. In the numerical simulations, we test the performance of our proposed estimator on a benchmark example containing a random (jump) discontinuous coefficient, however  we refer to \cite{ABarth_AStein_2018,ABarth_AStein_2022} and references therein for more theoretical discussion.
\end{remark}


\section{Stochastic Galerkin discretization}\label{sec:sgd}

In this section, the infinite dimensional problem \eqref{eqn:m1} will be reduced to finite dimensional space by applying 
Karhunen--Lo\`{e}ve (KL) expansion and  stochastic discontinuous Galerkin (SDG) method. For the statement of the finite dimensional structure, we follow
the notation in \cite{PCiloglu_HYucel_2022,PCiloglu_HYucel_2023}.

To represent the random inputs, we use the well--known Karhunen--Lo\`{e}ve (KL) 
expansion \cite{KKarhunen_1947,MLoeve_1946}, which is characterized  for a random variable $z(\boldsymbol{x}, \omega)$ by
\begin{equation}\label{eqn:kl}
	z(\boldsymbol{x},\omega) = \overline{z}(\boldsymbol{x}) + \kappa_{z} \sum \limits_{k=1}^{\infty} \sqrt{\lambda_k}\phi_k(\boldsymbol{x})\xi_k(\omega),
\end{equation}
where $\xi:=\{\xi_1,\xi_2,\ldots\}$ are uncorrelated random  variables. The pair of eigenvalues and eigenfunctions, i.e.,  $\{\lambda_k,\phi_k\}$, is obtained by solving 
Fredholm equation with the kernel $\mathcal{C}_{z}$  given in \eqref{eqn:covb}. Eigenvalues have a descending order such as 
$\lambda_1 \geq  \ldots \geq  \lambda_k \geq \ldots >0$, while the eigenfunctions are mutually orthonormal  polynomials. 
In order to find approximations in the finite-setting for the statistical moments of the problem  \eqref{eqn:m1}, we  truncate the expansion \eqref{eqn:kl} as follows
\begin{equation}\label{eqn:kltrun}
	z(\boldsymbol{x},\omega) \approx z_{N}(\boldsymbol{x},\omega) := \overline{z}(\boldsymbol{x}) + \kappa_{z} \sum \limits_{k=1}^{N} \sqrt{\lambda_k}\phi_k(\boldsymbol{x})\xi_k(\omega),
\end{equation}
where the truncation number $N$  depends on the decay of the eigenvalues $\lambda_i$ 
since $\mathcal{V}_{z}= |\mathcal{D}|^{-1} \sum \limits_{i=1}^{\infty} \lambda_i$; see, e.g., \cite{OGErnst_CEPowell_DJSilvester_EUllmann_2009a}. Then, the 
truncation error resulting from the KL--expansion is equivalent to
\begin{equation}\label{eig_bound}
	\|z - z_{N} \|_{L^2(\Omega;L^{2}(\mathcal{D}))} = \left( \sum \limits_{i=N+1}^{\infty} \lambda_i \right)^{1/2}.
\end{equation}
The positivity of the truncated diffusivity coefficient $a_N(\boldsymbol{x},\omega)$ can be ensured by applying a 
similar  condition as given in Assumption~\ref{assump_random} 
\begin{equation}\label{assumption_truncatedD}
	\exists \, a_{min}, \, a_{max} >0, \quad \text{s.t.} \quad a_{min} \leq  a_{N}(\boldsymbol{x},\omega) \leq a_{max}
\end{equation}
for almost every  $(\boldsymbol{x},\omega) \in  \mathcal{D} \times \Omega$; see also \cite{CEPowell_HCElman_2009} for a detailed discussion.


Then, based on the finite dimensional noise \cite[Assumption~2.1]{IBabuska_RTempone_GEZouraris_2004a}  and Doob--Dynkin Lemma \cite{BOksendal_2003}, we seek for the 
solution $u(\boldsymbol{x}, \omega) = u(\boldsymbol{x}, \xi_1(\omega), \xi_2(\omega), \ldots, \xi_{N}(\omega))$ on the probability space 
$(\Gamma, \mathcal{B}(\Gamma), \rho(\xi)d\xi)$, where $\Gamma= \prod \limits_{n=1}^{N} \Gamma_n$ is the support of probability density in the
finite dimensional space, $\mathcal{B}(\Gamma)$ is a Borel $\sigma$--algebra, and $\rho(\xi)d\xi$  represents the distribution measure of the vector $\xi$.
To construct an orthonormal basis space in  $L^2(\Gamma)$, we first introduce the 
set of finitely supported  multi--indices (or called as index set) \cite{ABespalov_CEPowell_DSilvester_2014,MEigel_CJGittelson_CSchwab_EZander_2014}
\[
\mathfrak{U} = \{ q = (q_1,q_2, \ldots) \in \mathbb{N}\; : \; | \text{supp} \,q| < \infty\},
\]
where $\text{supp}\,q = \{ n \in \mathbb{N} \; | \; q_n \neq 0\}$ and $|q| := \sum \limits_{i \in \text{supp}\, q} q_i$. Let $\{ \psi_{n}^{m}  \}_{m \in \mathbb{N}_0}$ 
denote the set of univariate polynomials of degree $m$ on $ \Gamma_n$ that are orthogonal basis of the space $ L^2(\Gamma_n)$. Then, 
the countable set of tensor product polynomials $\Psi^q(\xi)$, i.e.,
\[
\Psi^q(\xi) = \prod \limits_{n \in \text{supp}\,q} \psi^{q_n}_n(\xi_n),   \qquad \hbox{for} \;\, q \in \mathfrak{U}, \; \xi \in \Gamma,
\]
forms  an orthonormal basis for the space $L^2(\Gamma)$ \cite[Theorem~9.55]{GJLord_CEPowell_TShardlow_2014}.  Now, for a given finite index set 
$\mathfrak{B} \subset \mathfrak{U}$,  we can set the finite dimensional subspace  $\mathcal{Y}^q$   
with  $\text{dim} (\mathcal{Y}^q) = \# \mathfrak{B}$ 
\begin{eqnarray}
	\mathcal{Y}^q := \text{span} \{ \Psi^q: \; q \in \mathfrak{B} \}  \subset L^2(\Gamma),
\end{eqnarray}
and a subspace  $\mathcal{Y}^{q_n}\subset \mathcal{Y}^q$ for $n=1,\ldots, N$ consisting of the  family of polynomials $\{\widetilde{\Psi}^q\}_{q \in \mathfrak{B}}$ is defined by
\[
\widetilde{\Psi}^q(\xi) = \prod \limits_{i=1}^N \widetilde{\psi}^{q_i}_i(\xi_i) , \quad \hbox{where} \quad    
\widetilde{\psi}^{q_i}_i(\xi_i) = \begin{cases}
	\widetilde{\psi}^{q_n}_n(\xi_n), & i = n,\\
	1, & i \neq n.
\end{cases}
\]
Hence, the discrete parametric (stochastic)
solution, i.e., $u(\boldsymbol{x},\xi) \in \mathcal{Y}^q$, obtained by a (generalized) polynomial chaos (PC) approximation \cite{DXiu_GEKarniadakis_2002} is represented by 
\begin{equation}\label{eq:pcekdv}
	u^{q}(\boldsymbol{x},\xi) = \sum_{q \in \mathfrak{B} \subset \mathfrak{U}}u^q(\boldsymbol{x}) \Psi^q(\xi), \\
\end{equation}
where the deterministic modes are obtained by
\[
u^q(\boldsymbol{x}) = \frac{ \int_{\Gamma} u^{q}(\boldsymbol{x},\xi) \Psi^q(\xi) \,  \rho(\xi)d\xi}{\int_{\Gamma} (\Psi^q)^2(\xi)  \,  \rho(\xi)d\xi}.
\]


\begin{remark}\label{sto_basis}
	In general,  the choice of orthogonal polynomials, i.e., $\Psi^q$,  depends on the type of random field distribution 
	since the probability density functions of random distributions are equivalent to the weight functions 
	of certain types of orthogonal polynomials. For example,  Hermite polynomials corresponds to Gaussian 
	random variables, whereas  Legendre polynomials are chosen for the uniform random 
	variables; see, e.g.,  \cite{RKoekoek_PALesky_2010} for a detailed discussion.
\end{remark}

Last, we discretize the spatial domain  $\mathcal{D}$ by using the discontinuous Galerkin (DG) finite element 
spaces. Let $\{ \mathcal{T}_h\}_h$ be a decomposition of $\mathcal{D}$  into shape regular triangles $K$, having
the diameter $h_K$ and the edge length $h_{E}$. The set of interior edges  
and of  boundary edges are denoted by $\mathcal{E}^0_h$ and  $\mathcal{E}^{\partial}_h$, respectively, with  $\mathcal{E}_h=\mathcal{E}^{0}_h \cup\mathcal{E}^{\partial}_h$. Due to the discontinuity across
the edges $E \in K \cap K^e$, we set up the jump  $\jump{\cdot}$ and average $\average{\cdot}$ operators as follow
\begin{align*}
	&\jump{u}=u|_E\mathbf{n}_{K}+u^e|_E\mathbf{n}_{K^e}, \qquad \qquad \;\;\; \average{u}=\frac{1}{2}\big( u|_E+u^e|_E \big), \\
	&\jump{\nabla u}=\nabla u|_E \cdot \mathbf{n}_{K}+\nabla u^e|_E \cdot \mathbf{n}_{K^e}, \quad	\average{\nabla u}=\frac{1}{2}\big(\nabla u|_E+\nabla u^e|_E \big),
\end{align*}
for a piecewise continuous scalar function $u$  and a piecewise continuous vector field $\nabla u$ across the edge $E$, respectively. For $E \in \mathcal{E}^{\partial}_h$,
we further have $\average{\nabla u}=\nabla u$ and $\jump{u}=u\mathbf{n}$.
Then, by using the symmetric interior penalty Galerkin (SIPG) method with upwinding for the convection term
\cite{DNArnold_FBrezzi_BCockburn_LDMarini_2002a,BRiviere_2008a},  the (bi)--linear forms on the spatial domain for a fixed finite dimensional vector $\xi$ are given by
\begin{subequations}\label{bilinear}
	\begin{align}
		a_h(u,v,\xi)=& \sum \limits_{K \in \mathcal{T}_h} \int \limits_{K} \big( a_N(.,\xi) \nabla u \cdot  \nabla v + \mathbf {b}_N(.,\xi) \cdot \nabla u \,v \big)\, dx  \nonumber\\
        &-  \sum \limits_{ E \in \mathcal{E}_h } \int \limits_E \big( \average{a_N(.,\xi)\nabla u }  \jump{v} + \average{a_N(.,\xi) \nabla v }  \jump{u} \big)\, ds \nonumber \\
		&+ \hspace{-2mm} \sum \limits_{ E \in \mathcal{E}_h }  \frac{\sigma a_N(.,\xi) }{h_E} \hspace{-0.5mm} \int \limits_E \jump{u} \cdot \jump{v} \, ds
		+ \hspace{-1mm}\sum \limits_{K \in \mathcal{T}_h}  \int \limits_{\partial K^{-} \backslash \partial \mathcal{D}} \hspace{-3mm}\mathbf {b}_N(.,\xi) \cdot \mathbf{n}_E (u^e-u)v \, ds \nonumber \\
      &-  \sum \limits_{K \in \mathcal{T}_h} \; \int \limits_{\partial K^{-} \cap \partial \mathcal{D}^{-}} \mathbf {b}_N(.,\xi) \cdot \mathbf{n}_E \,u \,  v  \, ds
	\end{align}
	and
	\begin{align}
		l_h(v,\xi)=&\sum \limits_{K \in \mathcal{T}_h} \int \limits_{K} f v \, dx
		+ \sum \limits_{E \in \mathcal{E}_h^{\partial}} \int \limits_E  \left(\frac{\sigma a_N(.,\xi)}{h_E} \jump{v}  - \average{a_N(.,\xi)\nabla v} \right) u^d \, ds  \nonumber \\
		&- \sum \limits_{K \in \mathcal{T}_h}  \int \limits_{\partial K^{-} \cap \partial \mathcal{D}^{-}} \mathbf{b}_N(.,\xi) \cdot \mathbf{n}_E  \, u^d \, v  \, ds.
	\end{align}
\end{subequations}
Here, $\partial \mathcal{D}^-$ is the  inflow part of the boundary $\partial \mathcal{D}$ given by
\[
\partial \mathcal{D}^- = \set{\boldsymbol{x} \in \partial \mathcal{D}}{ \mathbf{b}_N(\boldsymbol{x},\xi) \cdot \mathbf{n}(\boldsymbol{\boldsymbol{x}}) < 0},
\]
whereas $\partial \mathcal{D}^+ = \partial \mathcal{D} \backslash \partial \mathcal{D}^-$ is the outflow part of $\partial \mathcal{D}$. Analogously,  
the inflow and outflow boundaries of an element $K \in \mathcal{T}_h$  are denoted by $\partial K^-$ and $\partial K^+$, respectively. The constant $\sigma>0$
represents the interior penalty parameter for DG formulation; see, \cite[Sec.~2.7.1]{BRiviere_2008a} for a detailed discussion. Further, denoting the set of all linear polynomials on $K$ by $\mathbb{P}(K)$,  the discontinuous
finite element spaces of trial and test variables are defined by
\begin{align}\label{tspace}
	V_h &= \set{v \in L^2(\mathcal{D})}{v\mid_{K}\in \mathbb{P}(K) \quad \forall K \in \mathcal{T}_h}.
\end{align}

Now, we are ready to state the variational formulation of \eqref{eqn:m1} obtained by stochastic discontinuous Galerkin discretization: Find $u_{h}^q \in V_h \otimes \mathcal{Y}^q$ such that
\begin{equation}\label{eq:bilinear}
	a_{\xi}(u_{h}^q,v_{h}^q) = l_{\xi}(v_{h}^q) \qquad  \forall v_{h}^q \in V_h \otimes \mathcal{Y}^q,
\end{equation}
where
\begin{equation*}
	a_{\xi}(u_{h}^q,v_{h}^q)=\int_{\Gamma}  a_h(u_{h}^q,v_{h}^q,\xi)\rho(\xi) \, d\xi, \quad 	l_{\xi}(v_{h}^q)=\int_{\Gamma}  l_h(v_{h}^q,\xi)\rho(\xi) \, d\xi.
\end{equation*}
Due the jump terms $\jump{\cdot}$, the bilinear form $a_{\xi}(u,v)$  \eqref{eq:bilinear} is not well--defined for 
the functions $ u,v \in H^1(\mathcal{D}) \otimes \mathcal{Y}^q$. However,  if we write the bilinear form $a_{\xi}(u,v)$  \eqref{eq:bilinear}  as follows 
\begin{equation}\label{eq:bilinear2}
	a_{\xi}(u,v) = \widetilde{a}_{\xi}(u,v) + j_{\xi}(u, v)  
\end{equation}
where
\[
j_{\xi}(u, v) = - \int_{\Gamma}  \Big( \sum \limits_{ E \in \mathcal{E}_h } \int \limits_E \big( \average{a_N(.,\xi)\nabla u }  \jump{v} + \average{a_N(.,\xi) \nabla v }  \jump{u} \big)\, ds \Big) \, \rho(\xi) \, d\xi,
\]
then the bilinear form $\widetilde{a}_{\xi}(u,v)$ becomes well--defined for the functions $ u,v \in H^1(\mathcal{D}) \otimes \mathcal{Y}^q$ and the following relation  holds \cite{DSchotzau_LZhu_2009a}
\begin{equation}\label{eq:bilinear3}
	a_{\xi}(u,v) =  \widetilde{a}_{\xi}(u,v), \qquad u,v \in H^1(\mathcal{D}) \otimes \mathcal{Y}^q.
\end{equation}


\section{Residual--based error estimator}\label{sec:estimator}

In this section, we derive a reliable residual--based error estimator between the solution $u \in H^1(\mathcal{D}) \times L^2(\Gamma)$ and the discrete solution 
$u_{h}^q \in V_h \otimes \mathcal{Y}^q$ based on the following energy norm
\begin{eqnarray} \label{energynorm}
	\lVert u\rVert _{\xi}^2= \int_{\Gamma}  \lVert u(.,\xi)\rVert_{e}^2 \, \rho(\xi)\, d\xi,
\end{eqnarray}
where
\begin{align*}
	\lVert u \rVert_{e}=&\Bigg( \sum \limits_{K \in \mathcal{T}_h} \int \limits_{K} a_N \, (\nabla u)^2 \, dx  + \sum \limits_{ E \in \mathcal{E}_h } 
	\frac{\sigma a_N}{h_E} \int \limits_E \jump{u}^2 \, ds \\
    & +\frac{1}{2}\sum \limits_{ E \in \mathcal{E}_h^{\partial}}\int \limits_E |\mathbf {b}_N\cdot \mathbf{n}_E |u^2 ds
	+ \frac{1}{2}\sum \limits_{ E \in \mathcal{E}^{0}_h }\int \limits_E |\mathbf {b}_N\cdot \mathbf{n}_E|(u^e-u)^2 \, ds\Bigg)^{\frac{1}{2}}.
\end{align*}
Throughout this section, we note that the symbol $\lesssim$ is used to denote bounds that are valid up to positive constants independent of the local mesh sizes in both spatial 
and parametric domains and the penalty parameter $\sigma$, provided that $\sigma \geq 1$.

By following \cite{PGCiarlet_2002,LRScott_SZhang_1990a}, we first define an $L^2$--projection operator  $\Pi_{q}: L^2(\Gamma) \rightarrow \mathcal{Y}^q$ by
\begin{subequations}\label{eq:l2proj}
	\begin{eqnarray}
		(\Pi_q(\xi)-\xi,\vartheta)_{L^2(\Gamma)}=0 \qquad \quad \forall \vartheta \in \mathcal{Y}^q, \qquad \forall\xi \in L^2(\Gamma),\label{eq:l2projc}
	\end{eqnarray}
	and an local $L^2$--projection operator  $\Pi_{q_n}: L^2(\Gamma) \rightarrow \mathcal{Y}^{q_n}$ for $n=1,2, \ldots, N$ by
	\begin{eqnarray}
		(\Pi_{q_n}(\xi)-\xi,\vartheta)_{L^2(\Gamma)}=0 \qquad \quad \forall \vartheta \in \mathcal{Y}^{q_n}, \qquad \forall\xi \in L^2(\Gamma). \label{eq:l2projclocal}
	\end{eqnarray}
\end{subequations}
Setting $\vartheta = \Pi_{q}(\xi)$  and $\vartheta = \Pi_{q_n}(\xi)$ in \eqref{eq:l2projc} and \eqref{eq:l2projclocal}, respectively, and applying Cauchy--Schwarz inequality, one can easily show that
\begin{equation}\label{ineq:inter1}
	\lVert \Pi_{q}(\xi)\rVert_{L^2(\Gamma)} \leq  \lVert \xi \rVert_{L^2(\Gamma)} \quad \hbox{and} \quad
	\lVert \Pi_{q_n}(\xi)\rVert_{L^2(\Gamma)} \leq \lVert \xi \rVert_{L^2(\Gamma)}.
\end{equation}
Moreover, since $\mathcal{Y}^{q_n} \subset \mathcal{Y}^q$, we have
\begin{eqnarray}
	(\Pi_q(\xi)-\xi,\Pi_{q_n}(\xi))_{L^2(\Gamma)}=0 \qquad \quad  \forall\xi \in L^2(\Gamma).
\end{eqnarray}
Let 
\[
f_h, u^d_h, a_{h,N} \in V_h, \quad  \hbox{and} \quad  \mathbf{b}_{h,N} \in (V_h) ^2
\]
denote the piecewise polynomial approximations  to the right--hand side function $f$, the Dirichlet boundary condition $u^d$, and the truncated random coefficient 
functions $a_N$ and $\mathbf{b}_N$, respectively. In addition, the weights on the element $K$ and the edge $E$ are given, respectively, by 
\begin{subequations}\label{weights}
	\begin{equation}
		\rho_{h,K} = h_K^2/ \|a_N\|_{L^2(K)}, \qquad   \rho_{q,K}= 1/(N \|a_N\|_{L^2(K)}),
	\end{equation}
	\begin{equation}
		\rho_{h,E} = h_E/ \|a_N\|_{L^2(E)},   \qquad  \rho_{q,E} = 1/(N \|a_N\|_{L^2(E)}).
	\end{equation}
\end{subequations}
Then, we propose  the following total error estimator 
\begin{align}\label{totalerror}
	\eta_T   &= \Bigg(  \underbrace{\sum \limits_{ K \in \mathcal{T}_h } \eta_{h,K}^2}_{\eta_h^2}  
	+ \underbrace{\sum \limits_{ K \in \mathcal{T}_h } \eta_{\theta,K}^2}_{\eta_{\theta}^2} 
	+ \underbrace{\sum \limits_{ K \in \mathcal{T}_h } \eta_{q,K}^2}_{\eta_q^2} \Bigg)^{1/2},
\end{align}
where the spatial error estimator for each element $K \in \mathcal{T}_h$ is
\begin{align*}
	\eta_{h,K}^2 =& \int \limits_{\Gamma} \left( \rho_{h,K} \, \lVert f_h + \nabla \cdot (a_{h,N} \nabla u_h^q) - \mathbf {b}_{h,N} \cdot \nabla u_h^q  \rVert_{L^2(K)}^2  \right) \rho(\xi)d\xi \\
	& + \int \limits_{\Gamma}  \sum \limits_{ E\in \partial K \backslash \partial \mathcal{D} } \Big( \rho_{h,E} \, \lVert \jump{a_{h,N}\nabla u_h^q } \rVert^2_{L^2(E)}\Big)\;  \rho(\xi)d\xi  \\
	& + \int \limits_{\Gamma}   \sum \limits_{E\in \partial K \backslash \partial \mathcal{D}} \left(\frac{\sigma}{h_E} \|a_{h,N}\|_{L^2(E)} 
	+ \|\mathbf{b}_{h,N}\|_{L^2(E)} + \rho_{h,E} \, \|\mathbf{b}_{h,N}\|_{L^2(E)}^2 \right) \,\\
	& \hspace{2cm} \times \lVert \jump{ u_h^q } \rVert^2_{L^2(E)} \; \rho(\xi)d\xi\\
	& + \int \limits_{\Gamma} \sum \limits_{ E \in \partial K \cap \partial \mathcal{D} }  \left(\frac{\sigma }{h_E} \|a_N\|_{L^2(E)}
	+ \|\mathbf{b}_N\|_{L^2(E)}\right)\lVert  u_h^q - u^d_{h} \rVert^2_{L^2(E)}\; \rho(\xi)d\xi,
\end{align*}
data approximation term emerged from  the discontinuity of the functions or parameters  is
\begin{align*}
	\eta_{\theta,K}^2 =&   \int \limits_{\Gamma} \left(  \rho_{h,K} \,  \Big(\lVert f-f_h  \rVert_{L^2(K)}^2+ \lVert\nabla \cdot ((a_N-a_{h,N}) \nabla u_h^q) \rVert_{L^2(K)}^2  \right. \\
    &  \left. \hspace{2cm} + \lVert (\mathbf {b}_N -\mathbf {b}_{h,N}) \cdot \nabla u_h^q  \rVert_{L^2(K)}^2\Big) \right) \rho(\xi)d\xi \\
	& + \int \limits_{\Gamma} \Big(  \sum \limits_{ E\in \partial K \backslash \partial \mathcal{D} } \rho_{h,E} \, \lVert \jump{(a_N-a_{h,N})\nabla u_h^q } \rVert^2_{L^2(E)}\; \Big) \rho(\xi)d\xi  \\
	& + \int \limits_{\Gamma}   \sum \limits_{E \in \partial K \backslash \partial \mathcal{D}} \hspace{-1mm}  \Big( \frac{\sigma}{h_E} \|a_N-a_{h,N}\|_{L^2(E)} + \|\mathbf{b}_N - \mathbf{b}_{h,N}\|_{L^2(E)}  \\
    &   \hspace{3cm}+ \rho_{h,E}  \, \|\mathbf{b}_N - \mathbf{b}_{h,N}\|_{L^2(E)}^2 \Big)\lVert \jump{ u_h^q } \rVert^2_{L^2(E)}  \; \rho(\xi)d\xi \\
	&  + \int \limits_{\Gamma}  \sum \limits_{ E \in \partial K \cap \partial \mathcal{D} } \hspace{-1mm} \Big(\frac{\sigma}{h_E} \|a_N\|_{L^2(E)} 
	+ \|\mathbf{b}_N\|_{L^2(E)}  \Big)\lVert  u^d_h - u^d \rVert^2_{L^2(E)}\;  \rho(\xi)d\xi,
\end{align*}
and the parametric error estimator is equivalent to
\begin{align*}
	\eta_{q,K}^2 = &   \int \limits_{\Gamma} \rho_{q,K} \, \sum \limits_{ n=1 }^N  \lVert \Pi_{q_n}(a_N \nabla u_h^q) - a_N \nabla u_h^q \rVert_{L^2(K)}^2 \,  \rho(\xi) \,d\xi  \\
    &+ \int \limits_{\Gamma} \rho_{q,K} \, \sum \limits_{ n=1 }^N  \lVert \Pi_{q_n}(\mathbf {b}_N \cdot \nabla u_h^q)- \mathbf {b}_N \cdot \nabla u_h^q \rVert_{L^2(K)}^2 \,  \rho(\xi) \,d\xi \nonumber \\
	& +  \int \limits_{\Gamma}  \sum \limits_{ E\in \partial K \backslash \partial \mathcal{D} } \rho_{q,E} \, \sum \limits_{ n=1 }^N   
	\lVert \Pi_{q_n}\left( \mathbf{b}_N \cdot \jump{ u_h^q}\right) - \mathbf{b}_N \cdot \jump{ u_h^q}\rVert_{L^2(E)}^2 \,  \rho(\xi) \,d\xi. \nonumber
\end{align*}
Before the derivation of reliability estimate for the proposed error estimates in \eqref{totalerror}, we construct  an operator 
$\mathfrak{T}_h : V_h  \otimes \mathcal{Y}^q \rightarrow H^1 \otimes \mathcal{Y}^q $, following the discussion in  
\cite[Theorem~2.1]{OAKarakashian_FPascal_2007} for the deterministic models,  satisfying  $ \mathfrak{T}_h v |_{\partial \mathcal{D}}  = \tilde{v}, \;  \forall v \in V_h  \otimes \mathcal{Y}^q$  and
\begin{subequations}\label{ineq:conform}
	\begin{align}
		\int \limits_{\Gamma} \hspace{-1.5mm} \sum \limits_{K \in \mathcal{T}_h}  \lVert v -  \mathfrak{T}_h v \rVert_{L^2(K)}^2 \; \rho(\xi) d\xi 
		& \lesssim  \int \limits_{\Gamma} \sum \limits_{ E \in \mathcal{E}^{0}_h }  h_E \,\lVert \jump{v} \rVert_{L^2(E)}^2  \; \rho(\xi) d\xi \nonumber\\
        & \quad+  \int \limits_{\Gamma} \sum \limits_{ E \in  \mathcal{E}_h^{\partial}}  h_E \,\lVert v -\tilde{v} \rVert_{L^2(E)}^2 \;\rho(\xi) d\xi, \\
		\int \limits_{\Gamma}  \hspace{-1.5mm} \sum \limits_{K \in \mathcal{T}_h}  \lVert \nabla(v -  \mathfrak{T}_h v) \rVert_{L^2(K)}^2 \; \rho(\xi) d\xi 
		& \lesssim  \int \limits_{\Gamma} \sum \limits_{ E \in \mathcal{E}^{0}_h }  h_E^{-1} \,\lVert \jump{v} \rVert_{L^2(E)}^2  \; \rho(\xi) d\xi  \nonumber \\
        &\quad +  \int \limits_{\Gamma} \sum \limits_{ E \in  \mathcal{E}_h^{\partial}}  h_E^{-1} \,\lVert v -\tilde{v} \rVert_{L^2(E)}^2 \;\rho(\xi) d\xi.	
	\end{align}
\end{subequations}
Moreover, a Cl\'{e}ment type interpolation $ I_h : H_0^1(\mathcal{D}) \rightarrow V_h^c$, where $V_h^c = V_h \cap H_0^1(\mathcal{D})$ is a conforming subspace of $V_h$,  
satisfies for $0 \leq k \leq l \leq 2$ \cite{ZCai_SZhang_2009,RVerfurth_1996}:
\begin{subequations}\label{ineq:ScottZhang}
	\begin{align}
		\lVert \nabla^k (v - I_hv) \rVert_{L^2(K)} &\lesssim h_K^{l-k} \lVert \nabla^{l} v \rVert_{L^2(\Delta_K)}, \label{ineq:ScottZhanga} \\
		\lVert v - I_hv \rVert_{L^2(E)} &\lesssim h_E^{1/2} \lVert \nabla v \rVert_{L^2(\Delta_E)}, \label{ineq:ScottZhangb}
	\end{align}
\end{subequations}
where $\Delta_K$ and $\Delta_E$ are the union of elements that share at least one vertex with the element $K$ and  the edge $E$, respectively. By the 
inequalities \eqref{ineq:inter1} and \eqref{ineq:ScottZhang}, for  $\forall v \in L^2(\Gamma;H^1(\mathcal{D}))$, we further have
\vspace{-2mm}
\begin{align}\label{ineqn:prjbddK}
	& \hspace{-1mm} \|\Pi_q(v-I_h v)\|_{L^2(\Gamma;L^2(K))}^2  \lesssim  \|v - I_h v\|_{L^2(\Gamma;L^2(K))}^2
	\lesssim     h_K^2  \|\nabla v\|_{L^2(\Gamma;L^2(\Delta_K))}^2
\end{align}
and
\vspace{-2mm}
\begin{align}\label{ineqn:prjbddE}
	& \hspace{-1mm} \|\Pi_q(v-I_h v)\|_{L^2(\Gamma;L^2(E))}^2  \lesssim  \|v - I_h v\|_{L^2(\Gamma;L^2(E))}^2
	\lesssim     h_E  \|\nabla v\|_{L^2(\Gamma;L^2(\Delta_E))}^2.
\end{align}



We can now find an upper bound for the error between the solution $u \in H^1(\mathcal{D}) \times L^2(\Gamma)$   and the discrete solution $u_{h}^q \in V_h \otimes \mathcal{Y}^q$ obtained 
from the stochastic discontinuous Galerkin, which states the reliability of the proposed estimator in \eqref{totalerror}. To derive the reliability estimate, we follow 
the standard techniques, such as the decomposition of the discontinuous discrete solution in the spatial domain into a conforming part plus a remainder 
in the spatial domain, and the projection operators in the spatial and parametric domains. 
Compared to the studies in the deterministic setting, for instance, \cite{PHouston_DSchotzau_TPWihler_2007,DSchotzau_LZhu_2009a}, 
we obtain additional error contributions due to the projection into parametric space.

\begin{theorem}\label{thm:reliable}
	Let $u$ be the continuous solution of \eqref{eqn:m1} on $\Gamma$  and  $u_h^q$ be its SDG approximation \eqref{eq:bilinear}. For the estimator $\eta_T$ given in \eqref{totalerror}, we have 
	\vspace{-1mm}
	\begin{align*}	
		\lVert u - u_h^q \rVert_{\xi} \lesssim  \eta_T.
	\end{align*}
\end{theorem}
\begin{proof}
	Split the discretized solution  $u_h^q$ into a conforming part plus a remainder in the spatial domain, i.e.,
	$
	u_h^q = u_{h,c}^q +u_{h,r}^q,
	$
	where $u_{h,c}^q = \mathfrak{T}_h u_h^q \in V_h^c  \otimes \mathcal{Y}^q$. Then, by the triangle inequality, we have
	\vspace{-1mm}
	\begin{align}\label{prf1}
		\lVert u - u_h^q \rVert_{\xi} \leq 	\lVert u_{h,r}^q \rVert_{\xi} + \lVert u - u_{h,c}^q  \rVert_{\xi}.
	\end{align}
	First, we find an upper bound for the remainder term in \eqref{prf1} in terms of the error estimator \eqref{totalerror}. 
	By the fact that $\jump{ u_{h,r}^q } = \jump{u_h^q}$ and the definition of the energy norm \eqref{energynorm},  we have
	\begin{align}\label{prf2}
		\lVert u_{h,r}^q  \rVert_{\xi}^2 = 
		&  \int \limits_{\Gamma} \hspace{-1.5mm}\sum \limits_{K \in \mathcal{T}_h} \int \limits_{K} a_N (\nabla u_{h,r}^q )^2 \, dx \, \rho(\xi) d\xi
		+ \int \limits_{\Gamma} \hspace{-1.5mm} \sum \limits_{E \in \mathcal{E}_h} \hspace{-1mm} \frac{\sigma a_N}{h_E} \int \limits_E \jump{u_h^q}^2 \, ds \, \rho(\xi) d\xi   \nonumber \\
		&+ \int \limits_{\Gamma}  \frac{1}{2}\sum \limits_{ E \in \mathcal{E}^{0}_h }\int \limits_E |\mathbf {b}_N\cdot \mathbf{n}_E|((u_{h,r}^q )^e - u_{h,r}^q )^2 \, ds\, \rho(\xi) d\xi \nonumber \\
       &+  \int \limits_{\Gamma}  \frac{1}{2} \sum \limits_{E \in \mathcal{E}_h^{\partial}} \int \limits_E |\mathbf {b}_N\cdot \mathbf{n}_E| (u_{h,r}^q )^2 ds \, \rho(\xi) d\xi.
	\end{align}
	An application of Cauchy-Schwarz inequality, the  inequalities in \eqref{ineq:conform} with $u_{h,r}^q  = u_h^q - \mathfrak{T}_h u_h^q$, 
	adding/subtracting the data approximation terms, and Young's inequality into the first term in \eqref{prf2} yields
	\begin{align*}
		& \int \limits_{\Gamma}  \sum \limits_{K \in \mathcal{T}_h} \int \limits_{K} a_N (\nabla u_{h,r}^q )^2 \, dx \, \rho(\xi) d\xi  \nonumber \\
		& \; \lesssim  \int \limits_{\Gamma}  \frac{1}{\sigma}\sum \limits_{E \in \mathcal{E}^{0}_h }  \frac{\sigma}{h_E} \Big( \|a_{h,N}\|_{L^2(E)}
		+  \|a_N -a_{h,N}\|_{L^2(E)} \Big) \lVert \jump{u_h^q}\rVert^2_{L^2(E)} \, \rho(\xi) d\xi \nonumber \\
    	& \qquad +  \int \limits_{\Gamma}  \frac{1}{\sigma}\sum \limits_{E \in \mathcal{E}^{\partial}_h }   \frac{\sigma}{h_E} 
		\|a_N \|_{L^2(E)}  \Big( \lVert u_h^q - u^d_h\rVert^2_{L^2(E)} + \lVert  u^d_h -u^d \rVert^2_{L^2(E)} \Big)\, \rho(\xi) d\xi  \nonumber 
    \end{align*}
    \begin{align}\label{prf3}
		& \; \lesssim \frac{1}{\sigma} \Big( \sum \limits_{ K \in \mathcal{T}_h } \eta_{h,K}^2 + \sum \limits_{ K \in \mathcal{T}_h } \eta_{\theta,K}^2 \Big).
	\end{align}
	A similar upper bound is also obtained for the second term in \eqref{prf2}. For the remaining terms in \eqref{prf2}, Cauchy-Schwarz and Young's inequalities  produce
	\begin{align}\label{prf32}
		&\int \limits_{\Gamma}  \frac{1}{2}  \sum \limits_{ E \in \mathcal{E}_h  }\int \limits_E |\mathbf {b}_N\cdot \mathbf{n}_E| ((u_{h,r}^q )^e - u_{h,r}^q  )^2 \, ds\, \rho(\xi) d\xi \nonumber\\
		& \; \lesssim  \int \limits_{\Gamma}   \hspace{-1.5mm} \sum \limits_{ E \in \mathcal{E}_h^0  }  \left( \|\mathbf {b}_{h,N}\|_{L^2(E)} +
		\|\mathbf {b} - \mathbf {b}_{h,N}\|_{L^2(E)} \right) \lVert \jump{u_h^q}\rVert^2_{L^2(E)} \, \rho(\xi) d\xi \nonumber \\
		& \qquad + \int \limits_{\Gamma}   \hspace{-1.5mm} \sum \limits_{ E \in \mathcal{E}_h^{\partial}  }  \|\mathbf {b}_N\|_{L^2(E)}
		\Big( \lVert u_h^q - u^d_h\rVert^2_{L^2(E)} + \lVert  u^d_h -u^d \rVert^2_{L^2(E)} \Big)\, ds\, \rho(\xi) d\xi   \nonumber \\
		& \; \lesssim \sum \limits_{ K \in \mathcal{T}_h } \eta_{h,K}^2 + \sum \limits_{ K \in \mathcal{T}_h } \eta_{\theta,K}^2.
	\end{align}
	So combining \eqref{prf3} and \eqref{prf32}, we get
	\begin{equation}\label{prf4}
		\lVert u_{h,r}^q  \rVert_{\xi}^2 \lesssim \sum \limits_{ K \in \mathcal{T}_h } \eta_{h,K}^2 + \sum \limits_{ K \in \mathcal{T}_h } \eta_{\theta,K}^2.
	\end{equation}
	
	Next, we find an estimate for the second term in \eqref{prf1}. By the fact that $ u|_{\partial \mathcal{D}} = u_{h,c}^q |_{\partial \mathcal{D}} = u^d$  
	due to the construction of  $\mathfrak{T}_h$, we have $u-u_{h,c}^q  \in H_0^1(\mathcal{D}) \otimes \mathcal{Y}^q$. Then, 
	the following inf--sup condition holds for all $v \in L^2(\Gamma;H^1_0(\mathcal{D}))$ \cite[Lemma~4.4]{DSchotzau_LZhu_2009a}
	\begin{align}\label{prf5}
		\lVert u - u_{h,c}^q  \rVert_{\xi} &\lesssim  \sup\limits_{v \in L^2(\Gamma;H^1_0(\mathcal{D}))} \dfrac{ \widetilde{a}_{\xi}(u-u_{h,c}^q ,v)}{\lVert v \rVert _{\xi}}.
	\end{align}
	By the bilinear systems \eqref{eq:bilinear} and \eqref{eq:bilinear2}, Cl\'{e}ment type interpolation estimates \eqref{ineq:ScottZhang}, and the continuity of the bilinear form, we obtain
	\begin{align}
		 &\widetilde{a}_{\xi}(u - u_{h,c}^q,v) \; =  \widetilde{a}_{\xi}(u,v) -  \widetilde{a}_{\xi}(u_{h,c}^q,v) = 
		\int \limits_{\Gamma} \int \limits_{\mathcal{D}} f v \, dx \, \rho(\xi)d\xi -  \widetilde{a}_{\xi}(u_{h,c}^q,v) 
		\nonumber\\
  &= \int \limits_{\Gamma} \int \limits_{\mathcal{D}} f v \, dx \, \rho(\xi)d\xi -  \widetilde{a}_{\xi}(u_h^q,v) +  \widetilde{a}_{\xi}(u_{h,r}^q,v) \nonumber \\
		& \; = \int \limits_{\Gamma} \int \limits_{\mathcal{D}} f I_hv \, dx \, \rho(\xi)d\xi  + \int \limits_{\Gamma} \int \limits_{\mathcal{D}} f (v-I_hv) \, dx \, \rho(\xi)d\xi  
		-  \widetilde{a}_{\xi}(u_h^q,v) +  \widetilde{a}_{\xi}(u_{h,r}^q,v) \nonumber \\
		& \; = \int \limits_{\Gamma} \int \limits_{\mathcal{D}} f (v-I_hv) \, dx \, \rho(\xi)d\xi  -  \widetilde{a}_{\xi}(u_h^q,v-I_hv) 
		+  j_{\xi}(u_h^q, I_h v) +  \widetilde{a}_{\xi}(u_{h,r}^q,v) \nonumber
    \end{align}
    \begin{align}\label{prf6}
		& \; \leq \lVert u_{h,r}^q \rVert_{\xi} \lVert v \rVert_{\xi} +  j_{\xi}(u_h^q, I_h v) 
		+ \int \limits_{\Gamma} \int \limits_{\mathcal{D}} f (v-I_hv) \, dx \, \rho(\xi)d\xi  - \widetilde{a}_{\xi}(u_h^q,v-I_hv).
	\end{align}
	The term $j_{\xi}(u_h^q, I_h v)$  in \eqref{prf6} is equivalent to
	\[
	j_{\xi}(u_h^q, I_h v) = - \int_{\Gamma}  \sum \limits_{ E \in \mathcal{E}_h } \int \limits_E \average{a_N \nabla I_h v }  \jump{u_h^q} \, ds \, \rho(\xi) \, d\xi,
	\]
	since $I_h v \in V_h^c$. Then, with the help of Cauchy--Schwarz  and inverse estimate
	\[
	\lVert v \rVert_{L^2(E)}   \lesssim  h_K^{-1/2} \|v\|_{L^2(K)} \qquad \forall v \in V_h
	\]
	with $h_E \lesssim h_K$, adding/subtracting the data approximation terms, and definition of the interpolation operator $I_h v$,  we obtain
	\begin{align}\label{prf63}
		&j_{\xi}(u_h^q, I_h v)   \nonumber  \\
		&  \; \lesssim   \sum \limits_{E \in \mathcal{E}_h} \| a_N \nabla I_h v \|_{L^2(\Gamma; L^2(E))} \|\jump{u_h^q}\|_{L^2(\Gamma; L^2(E))} \nonumber  \\
		& \;  \lesssim   \sum \limits_{E \in \mathcal{E}_h}  h_E^{1/2} \|a_N^{1/2} \nabla I_h v\|_{L^2(\Gamma; L^2(E))}  
		\sum \limits_{E \in \mathcal{E}_h} \|a_N^{1/2}\|_{L^2(\Gamma; L^2(E))} h_E^{-1/2} \|\jump{u_h^q}\|_{L^2(\Gamma; L^2(E))} \nonumber \\
		& \;  \lesssim  \sum \limits_{K \in \mathcal{T}_h}  \|a_N^{1/2} \nabla I_h v\|_{L^2(\Gamma; L^2(K))} 
		\sum \limits_{E \in \mathcal{E}_h} \|a_N^{1/2}\|_{L^2(\Gamma; L^2(E))} h_E^{-1/2} \|\jump{u_h^q}\|_{L^2(\Gamma; L^2(E))} \nonumber \\
		& \;  \lesssim  \sigma^{-1} \left(\sum \limits_{ K \in \mathcal{T}_h } \eta_{h,K}^2 + \sum \limits_{ K \in \mathcal{T}_h } \eta_{\theta,K}^2  \right)^{1/2} \|v\|_{\xi}.
	\end{align}
	Adding/substracting the  $L^2$--projection operator  $\Pi_{q}$   with $v \in L^2(\Gamma;H^1_0(\mathcal{D}))$, the last term  in \eqref{prf6} is stated as follows 
	\begin{equation}\label{prf7}
		T = \sum \limits_{i=1}^4 A_i,
	\end{equation}
	where
	\begin{align*}
		A_1  =&  \int \limits_{\Gamma}    \sum \limits_{K \in \mathcal{T}_h} \int \limits_{K}  \Big(f (v -I_hv) - \Pi_{q}(a_N \, \nabla u_h^q) \cdot  \nabla (v -I_hv)  \Big) dx \, \rho(\xi) d\xi \\
        &- \int \limits_{\Gamma}    \sum \limits_{K \in \mathcal{T}_h} \int \limits_{K}   \Pi_q(\mathbf{b}_N \cdot \nabla u_h^q) (v-I_h v)  dx \, \rho(\xi) d\xi \\
		& -  \int \limits_{\Gamma}  \sum \limits_{K \in \mathcal{T}_h} \int \limits_{\partial K^{-} \backslash \partial \mathcal{D}}  
		\Pi_{q}((\mathbf {b}_N \cdot \mathbf{n}_E) ((u_h^q)^e-u_h^q)(v -I_hv) \, ds\, \rho(\xi) d\xi,  
    \end{align*}
    \begin{align*}                         
		A_2 =& \int \limits_{\Gamma}  \sum \limits_{K \in \mathcal{T}_h} \int \limits_{K}   \Big(\Pi_{q}(a_N \, \nabla u_h^q) - a_N \nabla u_h^q \Big) \cdot  \nabla (v -I_hv)  \, dx \, \rho(\xi) d\xi, \\       
		A_3 =& \int \limits_{\Gamma}  \sum \limits_{K \in \mathcal{T}_h} \int \limits_{K} \Big(\Pi_q(\mathbf{b}_N \cdot \nabla u_h^q) - \mathbf{b}_N \cdot \nabla u_h^q \Big) (v-I_h v)\, dx \, \rho(\xi) d\xi, \\
		A_4 =& \int \limits_{\Gamma}  \sum \limits_{K \in \mathcal{T}_h}   \int \limits_{\partial K^{-} \backslash \partial \mathcal{D}} \hspace{-3.5mm} 
		\Big(\Pi_{q}((\mathbf {b}_N \cdot \mathbf{n}_E) ((u_h^q)^e-u_h^q)) -(\mathbf {b}_N \cdot \mathbf{n}_E) ((u_h^q)^e-u_h^q) \Big) \,  \\
        & \hspace{2cm} \times (v -I_hv)   dx \, \rho(\xi) d\xi.
	\end{align*}
	With the help of the definition of $L^2$--projection operator $\Pi_{q}$ \eqref{eq:l2projc},  the fact that $\Pi_q f = f$, and the integration by parts over $\mathcal{D}$, we obtain
	\vspace{-2mm}
	\begin{align}\label{prf8}
		A_1
		= &   \underbrace{\int \limits_{\Gamma}  \sum \limits_{K \in \mathcal{T}_h} \int \limits_{K} \big( f + \nabla \cdot (a_N \nabla u_h^q) - \mathbf{b}_N \cdot \nabla u_h^q \big) \, 
			\Pi_q(v -I_hv) \, dx \, \rho(\xi) d\xi}_{A_{1,1}} \nonumber \\
		& - \underbrace{\int \limits_{\Gamma}  \sum \limits_{ E \in \mathcal{E}^{0}_h }  \int \limits_E a_N \nabla u_h^q \cdot \mathbf{n}_E \, \Pi_{q}(v-I_hv)   \, ds \, \rho(\xi) d\xi}_{A_{1,2}}   \nonumber \\
		& -  \underbrace{\int \limits_{\Gamma}  \sum \limits_{K \in \mathcal{T}_h} \int \limits_{\partial K^{-} \backslash \partial \mathcal{D}} 
			((\mathbf {b}_N \cdot \mathbf{n}_E) ((u_h^q)^e-u_h^q) \, \Pi_{q}(v -I_hv) \, ds\, \rho(\xi) d\xi}_{A_{1,3}}.
	\end{align}
	By adding/subtracting the data approximation terms and using Cauchy--Schwarz inequality with \eqref{energynorm} and \eqref{ineqn:prjbddK}, the first  term in \eqref{prf8}
	is bounded by
	\begin{subequations}\label{prf9}
		\begin{align}
			A_{1,1} =   &   \int \limits_{\Gamma}   \sum \limits_{K \in \mathcal{T}_h} \int \limits_{K} \big( f_h + \nabla \cdot (a_{h,N} \nabla u_h^q) - \mathbf{b}_{h,N} \cdot \nabla u_h^q \big) \,
			\Pi_q(v -I_hv) \, dx \, \rho(\xi) d\xi \nonumber \\  
			& +   \int \limits_{\Gamma}   \sum \limits_{K \in \mathcal{T}_h} \int \limits_{K} \big( (f-f_h) + \nabla \cdot ((a_N-a_{h,N}) \nabla u_h^q) - (\mathbf{b}_N -\mathbf{b}_{h,N}) \cdot \nabla u_h^q \big) \,  \nonumber\\
            & \hspace{2cm} \times\Pi_q(v -I_hv) \, dx \, \rho(\xi) d\xi \nonumber \\
			\lesssim &  \left( \sum \limits_{K \in \mathcal{T}_h} \eta_{h,K}^2 + \eta_{\theta,K}^2 \right)^{1/2} \|v\|_{\xi}.
		\end{align}
		Analogously, an application of the Cauchy--Schwarz, the inequality \eqref{ineqn:prjbddE}, and  the definition of energy norm \eqref{energynorm} yields
		\begin{align}
			A_{1,2}  \lesssim   &  \sum \limits_{E \in \mathcal{E}^0}   \|\jump{ a_N \,\nabla u_h^q} \|_{L^2(\Gamma; L^2(E))} \, \|\Pi_{q}(v -I_hv)\|_{L^2(\Gamma; L^2(E))} \nonumber \\
			\lesssim    &  \sum \limits_{E \in \mathcal{E}^0}   \|\jump{((a_N - a_{h,N}) + a_{h,N} ) \nabla u_h^q} \|_{L^2(\Gamma; L^2(E))} \, h_E^{1/2} \|\nabla v\|_{L^2(\Gamma;L^2(\Delta_E))} \nonumber \\
			&  \lesssim      \left( \sum \limits_{K \in \mathcal{T}_h} \eta_{h,K}^2 + \eta_{\theta,K}^2 \right)^{1/2} \|v\|_{\xi}, \\     
			A_{1,3}  \lesssim   &  \sum \limits_{E \in \mathcal{E}^0}   \| \mathbf {b}_N \cdot \jump{u_h^q} \|_{L^2(\Gamma; L^2(E))} \, \|\Pi_{q}(v -I_hv)\|_{L^2(\Gamma; L^2(E))} \nonumber \\
			\lesssim    &  \sum \limits_{E \in \mathcal{E}^0}   \| \big( (\mathbf {b}_N-\mathbf{b}_{h,N}) + \mathbf {b}_{h,N} \big) \cdot \jump{u_h^q} \|_{L^2(\Gamma; L^2(E))}  
			\, h_E^{1/2} \|\nabla v\|_{L^2(\Gamma;L^2(\Delta_E))}   \nonumber \\
			\lesssim    &   \left( \sum \limits_{K \in \mathcal{T}_h} \eta_{h,K}^2 + \eta_{\theta,K}^2 \right)^{1/2} \|v\|_{\xi}.
		\end{align}
	\end{subequations}
	Now, we bound the rest of terms in \eqref{prf7} in terms of the parametric estimator given in \eqref{totalerror}. 
	By $\mathcal{Y}^{q_n} \subset \mathcal{Y}^q \subset L^2(\Gamma)$ and \eqref{eq:l2proj},  we have for each $n=1, \ldots N$
	\begin{align*}
		& \lVert \Pi_{q}(a_N \, \nabla u_h^q)- a_N \, \nabla u_h^q  \rVert_{L^2(\Gamma; L^2(\mathcal{D}))}  \\
        & \leq   \underbrace{\lVert \Pi_{q}(a_N  \nabla u_h^q)- \Pi_{q_n}(a_N  \nabla u_h^q) \rVert_{L^2(\Gamma; L^2(\mathcal{D}))}}_{=0} + \lVert \Pi_{q_n}(a_N  \nabla u_h^q)- a_N  \nabla u_h^q \rVert_{L^2(\Gamma; L^2(\mathcal{D}))}.
	\end{align*}
	Then,
	\begin{equation}\label{prf10}
		N \lVert \Pi_{q}(a_N \nabla u_h^q)- a_N \nabla u_h^q  \rVert_{L^2(\Gamma; L^2(\mathcal{D}))} 
		\leq \sum \limits_{n=1}^N \lVert \Pi_{q_n}(a_N \nabla u_h^q)- a_N \nabla u_h^q \rVert_{L^2(\Gamma; L^2(\mathcal{D}))}.
	\end{equation}
\noindent Hence, the Cauchy--Schwarz inequality, the inequalities  \eqref{ineq:ScottZhang} and \eqref{prf10} with $h_E \leq h_K < 1$ give us
	\begin{subequations}\label{prf11}
		\begin{eqnarray}
			A_2 &=& \int \limits_{\Gamma}   \sum \limits_{K \in \mathcal{T}_h} \int \limits_{K} \hspace{-1mm}  \big(\Pi_{q}(a_N \nabla u_h^q ) - a_N \nabla u_h^q \big) \cdot  \nabla (v -I_hv)  \, dx \, \rho(\xi) d\xi \nonumber \\
			&\lesssim & \frac{1}{N}\sum \limits_{n=1}^N \lVert \Pi_{q_n}(a_N \nabla u_h^q)- a_N\nabla u_h^q \rVert_{L^2(\Gamma; L^2(\mathcal{D}))} \lVert \nabla v \rVert_{L^2(\Gamma; L^2(\mathcal{D}))} \nonumber \\ 
			&\lesssim & \frac{1}{N}\sum \limits_{n=1}^N \lVert \Pi_{q_n}(a_N \nabla u_h^q)- a_N \nabla u_h^q \rVert_{L^2(\Gamma; L^2(\mathcal{D}))} \, \|a_N^{-1/2}\|_{L^2(\Gamma; L^2(\mathcal{D}))} \|v\|_{\xi} \nonumber\\
			&\lesssim &   \left( \sum \limits_{K \in \mathcal{T}_h} \eta_{q,K}^2 \right)^{1/2} \,  \|v\|_{\xi},  
\end{eqnarray}
\begin{eqnarray} 
			A_3 &=& \int \limits_{\Gamma}  \sum \limits_{K \in \mathcal{T}_h} \int \limits_{K} \big(\Pi_q(\mathbf{b}_N \cdot \nabla u_h^q) - \mathbf{b}_N \cdot \nabla u_h^q \big) (v-I_h v)\, dx \, \rho(\xi) d\xi \nonumber \\
			&\lesssim &  \sum \limits_{K \in \mathcal{T}_h}\lVert \Pi_q(\mathbf{b}_N \cdot \nabla u_h^q) - \mathbf{b}_N \cdot \nabla u_h^q  \rVert_{L^2(\Gamma; L^2(K))} \, h_K \,\lVert \nabla v \rVert_{L^2(\Gamma;L^2(K))} \nonumber \\
			&\lesssim &  \frac{1}{N}\sum \limits_{n=1}^N \lVert \Pi_{q_n}(\mathbf{b}_N \cdot \nabla u_h^q) - \mathbf{b}_N \cdot \nabla u_h^q  \rVert_{L^2(\Gamma;L^2(\mathcal{D}))} \,
			\|a_N^{-1/2}\|_{L^2(\Gamma;L^2(\mathcal{D}))} \|v\|_{\xi} \nonumber\\
			&\lesssim &   \left( \sum \limits_{K \in \mathcal{T}_h} \eta_{q,K}^2 \right)^{1/2} \,  \|v\|_{\xi}, \\       
			A_4 &=& \int \limits_{\Gamma}  \sum \limits_{K \in \mathcal{T}_h}  \int \limits_{\partial K^{-} \backslash \partial \mathcal{D}} \hspace{-3mm} 
			\big(\Pi_{q}((\mathbf {b}_N \cdot \mathbf{n}_E) ((u_h^q)^e-u_h^q)) -(\mathbf {b}_N \cdot \mathbf{n}_E) ((u_h^q)^e-u_h^q) \big) \, \nonumber \\ 
            & & \hspace{2.5cm}\times (v -I_hv)   dx \, \rho(\xi) d\xi \nonumber \\
			&\lesssim & \sum \limits_{ E \in \mathcal{E}^0}  \| \Pi_q(\mathbf {b}_N   \cdot \jump{u_h^q}) - \mathbf{b}_N \cdot \jump{u_h^q} \|_{L^2(\Gamma;L^2(E))} \|v-I_hv\|_{L^2(\Gamma;L^2(E))}   \nonumber \\
			&\lesssim & \sum \limits_{ E \in \mathcal{E}^0}  \| \Pi_q(\mathbf {b}_N   \cdot \jump{u_h^q}) - \mathbf{b}_N \cdot \jump{u_h^q} \|_{L^2(\Gamma; L^2(E))} h_E^{1/2}  \|\nabla v\|_{L^2(\Gamma; L^2(\Delta_E))}   \nonumber\\
			&\lesssim &  \left( \sum \limits_{K \in \mathcal{T}_h} \eta_{q,K}^2 \right)^{1/2} \,  \|v\|_{\xi}.
		\end{eqnarray}
	\end{subequations}
	Inserting  the bounds  obtained in \eqref{prf4}, \eqref{prf63}, \eqref{prf9}, and \eqref{prf11} into \eqref{prf5}, we get
	\begin{equation}\label{prf12}
		\lVert u - u_h^c \rVert_{\xi} \lesssim \left( \sum \limits_{K \in \mathcal{T}_h} \eta_{h,K}^2 + \eta_{\theta,K}^2 + \eta_{q,K}^2 \right)^{1/2}.
	\end{equation}
	Last, we combine the results in \eqref{prf4} and \eqref{prf12} to get  the desired result.
\end{proof} 

\begin{remark}
	In our numerical implementations, we  propose an adaptive approach, based on  mesh refinement or parametric enrichment with polynomial degree adaption by following the upper bounds in Theorem~\ref{thm:reliable}. In order to increase the efficiency of the proposed approach, coarsening in the spatial domain or parametric index set can be imposed by considering a lower bound for the error in the a posteriori error analysis. By taking into account the simplicity of the underlying methodology, this issue will be addressed in the future studies.
\end{remark}


\section{Numerical experiments}\label{sec:num}
\subsection{The adaptive loop}\label{sec:loop}

A generic adaptive refinement/enrichment procedure for the numerical solution of the model problem  \eqref{eqn:m1} consists of successive loops of the sequence given 
in \eqref{loop}. Starting with a given triangulation $\mathcal{T}_h^0$ and an initial index set $\mathfrak{B}^0$, the adaptive procedure generates a sequence of 
triangulations $\mathcal{T}_h^{k} \subseteq \mathcal{T}_h^{k+1}$ and of index sets  $\mathfrak{B}^k \subseteq  \mathfrak{B}^{k+1}$. The \textbf{SOLVE} step (subroutine \textbf{solve})  corresponds to the numerical approximations of the statistical moments obtained by 
the stochastic discontinuous Galerkin numerical scheme in the current mesh and index set. We here solve the following linear system:
\begin{equation}\label{tensormtrx}
	\underbrace{\left( \sum_{i=0}^{N} \mathcal{G}_i \otimes \mathcal{K}_i \right)}_{\mathcal{A}} \, \mathbf{u} = 
	\underbrace{\left( \sum_{i=0}^{N} \mathbf{g}_i \otimes \mathbf{f}_i\right)}_{\mathcal{F}},
\end{equation}
where
$
\mathbf{u} =\left(
u_0, \ldots,u_{(\# (\mathfrak{B}^k)-1)}
\right)^T \;\; \hbox{with} \;\;  u_i \in \mathbb{R}^{N_d}, \;\; i=0,1,\ldots,(\# (\mathfrak{B}^k)-1)$ 
with  the degree of freedoms in the spatial discretization $N_d$.
Here, $\mathcal{K}_i \in \mathbb{R}^{N_{d} \times N_d}$  and  $\mathcal{G}_i \in \mathbb{R}^{(\# (\mathfrak{B}^k)) \times (\# (\mathfrak{B}^k))}$ 
represent the stiffness and stochastic matrices, respectively, whereas $\mathbf{f}_i \in \mathbb{R}^{N_d}$ and  $\mathbf{g}_i \in \mathbb{R}^{(\# (\mathfrak{B}^k))}$ 
are the right--hand side and stochastic vectors, respectively; see  \cite{PCiloglu_HYucel_2022} for the constructions of the matrices and vectors in details. 
In order to solve the matrix system \eqref{tensormtrx},  we use a generalized minimal residual (GMRES) solver in combination with the well-known 
mean-based preconditioner \cite{CEPowell_HCElman_2009}, that is,
\begin{equation*}
	\mathcal{P}_0=\mathcal{G}_0 \otimes \mathcal{K}_0,
\end{equation*}
where $ \mathcal{G}_0  $ and $ \mathcal{K}_0 $ are the mean stochastic and stiffness matrices, respectively. 
In the  \textbf{ESTIMATE} step (subroutine \textbf{estimate}), we compute the error estimator  
$\eta_T$ in \eqref{totalerror} contributed by the error due to the (generalized)  polynomial chaos discretization, the error due to the SIPG discretization 
of the (generalized)  polynomial chaos coefficients in the expansion, and the error due to the data oscillations. 
In the computation of spatial terms of 
the estimator, we just make Kronecker product of the 
spatial contributions with the stochastic matrix $\mathcal{G}_0$, whereas we use the stochastic mass matrix in the computation of parametric term  
by following the classical projection computation. In  \eqref{totalerror}, we replace the continuous coefficients $a_N$, $\mathbf{b}_N$  in the spatial domain 
by the discrete ones $a_{h,N}$, $\mathbf{b}_{h,N}$, respectively, to obtain a fully a posteriori error estimate. Recall that for a random coefficient $z(\boldsymbol{x},w)$, 
we have
\begin{eqnarray}\label{ineq:rand4}
	\|z -z_{h,N} \|_{L^2(\Gamma;L^2(\mathcal{D}))} &\leq&  \|z -z_{N} \|_{L^2(\Gamma;L^2(\mathcal{D}))} + \|z_{N} -z_{h,N} \|_{L^2(\Gamma;L^2(\mathcal{D}))} \nonumber \\
	&=& \left( \sum \limits_{i=N+1}^{\infty} \lambda_i \right)^{1/2} + \|z_{N} -z_{h,N} \|_{L^2(\Gamma;L^2(\mathcal{D}))}.
\end{eqnarray}
When the truncation error coming from KL expansion, that is, the first term in \eqref{ineq:rand4}, is ignored,  the spatial estimator may dominates the parametric estimator; see, e.g., 
\cite{ABespalov_LRocchi_2018,ABespalov_DSilvester_2016}. To avoid such a scenario and construct a balance between the errors emerged from spatial and stochastic domains, data oscillation term 
emerged from the truncation error of random coefficient on each element $K \in \mathcal{T}_h^k$, that is,
\begin{align}\label{eq:estKL}
	\eta_{z,K}^2 =&   \int \limits_{\Gamma} \left(  \rho_{h,K} \Big( \lVert\nabla \cdot (\Lambda_a \nabla u_h^q) \rVert_{L^2(K)}^2  
	+\lVert \Lambda_b \cdot \nabla u_h^q  \rVert_{L^2(K)}^2\Big) \right) \rho(\xi)d\xi \nonumber\\
	& + \int \limits_{\Gamma} \Big(  \sum \limits_{ E\in \partial K \backslash \partial \mathcal{D} } \rho_{h,E} \lVert \jump{\Lambda_a \nabla u_h^q } \rVert^2_{L^2(E)}\; \Big) \rho(\xi)d\xi  \nonumber \\
	& + \int \limits_{\Gamma}   \sum \limits_{E \in \partial K \backslash \partial \mathcal{D}} \hspace{-1mm}  \Big( \frac{\sigma \Lambda_a}{h_E} 
	+ \Lambda_b  + \rho_{h,E}  \Lambda_b^2  \Big)\lVert \jump{ u_h^q } \rVert^2_{L^2(E)}  \; \rho(\xi)d\xi
\end{align}
is added to the parametric estimator. In \eqref{eq:estKL}, for any random coefficient $z$, $\Lambda_z$  is computed by $\Lambda_z = \left( \sum \limits_{i=N+1}^{N_{\infty}} \lambda_i  \right)^{1/2}$, where $N_{\infty}$ is chosen as the KL expansion covering the 97\% of sum of the eigenvalues $\lambda_i$; see Table~\ref{tab:Ninfty} for varying the correlation length $\ell$.

\begin{table}[htp!]
	\caption{$N_{\infty}$ values for varying the correlation length $\ell$.}
	\label{tab:Ninfty}
	\centering
	\begin{tabular}{c|ccccc}
		$\ell $  & $ 0.25$    & $ 0.5$     & $ 1$       & $ 5$      & $ 10$    \\ \hline
		$N_{\infty}$        & $599$ & $370$ & $154$ & $15$ & $9$
	\end{tabular}
\end{table}

In the \textbf{MARK} step (subroutine \textbf{mark}), we specify the elements in $\mathcal{T}_h^k$ 
for refinement by using the  spatial estimators, that is,  $\eta_{h,K}$ and $\eta_{\theta,K}$, and by choosing a subset $\mathcal{M}_h^k \subset \mathcal{T}_h^k$ 
such that the following  bulk criterion  is satisfied for a fixed marking parameter $0 < \theta_h \leq 1$:
\begin{equation}\label{ineq:markphys}
	\theta_h 	\sum \limits_{ K \in \mathcal{T}_h^k } \big(\eta_{h,K}^2 + \eta_{\theta,K}^2 \big) \leq  
	\sum \limits_{ K \in \mathcal{M}_h^k \subset \mathcal{T}_h^k} \big(\eta_{h,K}^2 + \eta_{\theta,K}^2\big).
\end{equation}
On the other hand, to  build a minimal subset of marked indices, we first introduce a new index set \cite{MEigel_CJGittelson_CSchwab_EZander_2014}
\begin{equation}
	\mathfrak{R}^k := \{ q \in \mathfrak{U} \backslash \mathfrak{B}^k \, : \, q = q_{\mathfrak{B}} \pm \epsilon^{(n)} \;\,  
	\forall q_{\mathfrak{B}} \in \mathfrak{B}^k, \; \forall n=1, \ldots  N_{\mathfrak{B}} \},
\end{equation}
where the counter parameter is
\[
N_{\mathfrak{B}} =\left\{
\begin{array}{ll}
	0, & \hbox{if }  \mathfrak{B}= \{\mathbf{0}\}, \\
	\max\{\max(\text{supp}(q_{\mathfrak{B}}))\,: q_{\mathfrak{B}} \in \mathfrak{B}\backslash \{\mathbf{0}\}\}, & \hbox{otherwise},
\end{array}
\right.
\]
and $\epsilon^{(n)} = \big( \epsilon^{(n)}_1, \epsilon^{(n)}_2, \ldots   \big)$ is the Kronecker delta sequence satisfying $\epsilon^{(n)}_j = \delta_{nj}$ 
for all $j \in \mathbb{N}$. Then, by using the parametric error estimator $\eta_q$  for a fixed marking parameter $0 < \theta_q \leq 1$, we apply 
a bulk criterion in the parametric setting as follows
\begin{equation}\label{ineq:markpara}
	\theta_q \sum \limits_{ q \in \mathfrak{R}^k } \eta_{q}^2  \leq \sum \limits_{ q \in \mathfrak{M}^k \subseteq  \mathfrak{R}^k} \eta_{q}^2.
\end{equation}
Finally, the \textbf{REFINE} step (subroutine \textbf{refine}) either performs 
a refinement of the current triangulation or an enrichment of the current index set based on the dominant error estimator contributing 
to the total estimator. In the refinement of a triangulation, the marked elements are refined by longest edge bisection, whereas the elements 
of the marked edges are refined by bisection. On the other hand, the enrichment of the polynomial space is made by adding all marked indices 
to the current index, i.e., $\mathfrak{B}^{k+1} =  \mathfrak{B}^k \cup \mathfrak{M}^k$. Overall, the adaption process is summarized 
in Algorithm~\ref{alg:adap}, which repeats until the prescribed tolerance $TOL$  is met by the total estimator or the maximum number 
of total degree of freedoms $\max_{dof}$ is reached.

\begin{algorithm}[htp!]
	\scriptsize
	\caption{Adaptive stochastic discontinuous Galerkin algorithm}
	\label{alg:adap}
	\hspace*{\algorithmicindent}\textbf{Input:} initial mesh  $ \mathcal{T}_h^0 $, initial index set $ \mathfrak{B}^0$, marking parameters $\theta_h$ and $\theta_q$, data $f,u^d, a, \mathbf{b}$, tolerance threshold $TOL$, and maximum degree of freedom $\max_{dof}$. \\
	\vspace{-5mm}
	\begin{algorithmic}[1]
		\FOR{$ k = 0,1,2, \ldots $}
		\STATE{$  (u_h^q)^k = \textbf{solve}(\mathcal{T}_h^k, \mathfrak{B}^k, f, u^d, a, \mathbf{b})$}
		\STATE{$(\eta_h^2, \eta_{\theta}^2, \eta_q^2) = \textbf{estimate}((u_h^q)^k,\mathcal{T}_h^k, \mathfrak{B}^k, f, u^d, a, \mathbf{b})$}
		\IF{$ (\eta_T^2 \leq TOL)$ or  $(\# (\mathfrak{B}^k) \times N_d \leq \max_{dof})$}
		\STATE{\textbf{break;}}
		\ENDIF
		\STATE $\mathcal{M}_h^k =  \textbf{mark}(\mathcal{T}_h^k, \eta_h^2, \eta_{\theta}^2, \theta_h)$  and $\mathfrak{M}^k =  \textbf{mark}(\mathfrak{U}^k, \mathfrak{B}^k, \eta_q^2, \theta_q)$
		\IF{ $ (\eta_h^2 + \eta_{\theta}^2) \geq \eta_{q}^2 $}
		\STATE  $\mathcal{T}_h^{k+1} = \textbf{refine}(\mathcal{T}_h^k,\mathcal{M}_h^k)$
		\ELSE
		\STATE $\mathfrak{B}^{k+1} = \textbf{refine}(\mathfrak{B}^k,\mathfrak{M}^k)$
		\ENDIF
		\ENDFOR
	\end{algorithmic}
\end{algorithm}

\subsection{Numerical results}

To investigate the quality of the derived estimators in Section~\ref{sec:estimator} and the performance of the adaptive 
loop proposed in Section~\ref{sec:loop}, several benchmark  examples including a random diffusivity parameter, a random convectivity parameter, random diffusivity/convectivity parameters, and a random (jump) discontinuous diffusivity parameter, are provided in this section. We characterize the random coefficients by using exponential covariance function with the correlation length $\ell_n$  and the corresponding eigenpair $(\lambda_j, \phi_j)$ is given explicitly in \cite[pp. 295]{GJLord_CEPowell_TShardlow_2014}. Since the underlying random 
variables have chosen based on the uniform distribution  over $[-1,1]$, that is,  $ \xi_j \sim \mathcal{U}[-1,1]$ for $i=1,\ldots, N$, the stochastic basis functions are taken to be Legendre polynomials; see, Remark~\ref{sto_basis}. In the numerical implementations, the initial mesh and index set are  chosen as $ \mathcal{T}_h^0 $ with $N_d=384$ and $
\mathfrak{B}^0 = \{ (0,0,0,\ldots), (1,0,0,\ldots) \}$, respectively. Uniformly refined  meshes are constructed by dividing each triangle into four 
triangles, whereas uniformly refined index sets are generated by increasing the truncation number $N$ and polynomial degree in $\Psi$ by one unit. 
To test the effectivity of the derived estimator $\eta_T$ in \eqref{totalerror}, we compare the estimator with the reference error defined by
$
e^k_{ref} = u_{ref} - (u_h^q)^k,
$
where $u_{ref}$ is generated  by applying one step uniform mesh refinement to the final triangulation produced  from the Algorithm~\ref{alg:adap} and 
one adaptive enrichment  to the final index set. By using Galerkin orthogonality, we compute the effectivity indexes according to 
\begin{equation}\label{eff}
	\mathcal{I}_{eff} = \frac{\eta_T} {\|e^k_{ref}\|_{\xi}} = \frac{\eta_T}{\big(\|u_{ref}\|_{\xi} - \| (u_h^q)^k\|_{\xi}\big)^{1/2}}, \qquad \forall k \geq 0.
\end{equation}
Unless otherwise stated, in all simulations, we take the correlation length $\ell=1$ and standard deviation $\kappa=0.05$, and the Algorithm~\ref{alg:adap} is 
terminated when the total error estimator $\eta_T$ is reduced to $TOL = 1e-6$ or the  degrees of freedom is reached to maximum ($\max_{dof} = 10^7$). 
Further, we evaluate the performance of the our adaptive algorithm by measuring the total computational cost required to reach the stopping criteria. As defined in \cite{ABespalov_DPraetorius_LRocchi_MGuggeri_2019b, ABespalov_DPraetorius_LRocchi_MGuggeri_2019},  the computational cost is defined as the cumulative number of degrees of freedom over all iterations of the adaptive algorithm, that is,
\begin{equation}
	N_{cost} = \sum\limits_{k = 0}^{\# iter} \{\# \text{Total DOFs}\}^k,
\end{equation}
where $\{\# \text{Total DOFs}\}^k$ denotes the total number of degrees of freedom at the k-th iteration.

\subsubsection{Example with random diffusivity}\label{ex:rand_diff}

Our first test example is taken from \cite{KLee_HCElman_2017} and the data of the problem is as follows
\[
\mathcal{D} = [-1,1]^2, \quad f(\boldsymbol{x})=0, \quad \mathbf{b}(\boldsymbol{x})=(0,1)^T,
\]
with the nonhomogeneous Dirichlet boundary condition
\begin{equation}\label{nonDBC}
	u^d(\boldsymbol{x})=\begin{cases}
		u^d(x_1,-1)=x_1, & u^d(x_1,1)=0,\\
		u^d(-1,x_2)=-1, & u^d(1,x_2)=1.
	\end{cases}
\end{equation}
Here, the diffusion parameter $a(\boldsymbol{x}, \omega)$ is given in the form of $a(\boldsymbol{x}, \omega) = \nu z(\boldsymbol{x}, \omega)$, 
where $z(\boldsymbol{x}, \omega)$ is  a random variable having unity mean $\overline{z}(\boldsymbol{x}) =1$ and $\nu$ is the viscosity parameter. As $\nu$ decreases, 
the solution of the underlying problem  exhibits  boundary layer near $x_2 =1$, where the value of the solution changes in a dramatic way \cite{KLee_HCElman_2017}. 
Locally refined triangulations generated by the Algorithm~\ref{alg:adap} with the marking parameters $\theta_h=0.5$ and $\theta_q =0.5$ are given in  
Figure~\ref{Ex1_mesh}. It is observed that our estimator $\eta_T$   detects well the regions where the mean of the solution is not smooth enough.

\begin{figure}[htp!]
	\centering
	 \includegraphics[width=1\textwidth]{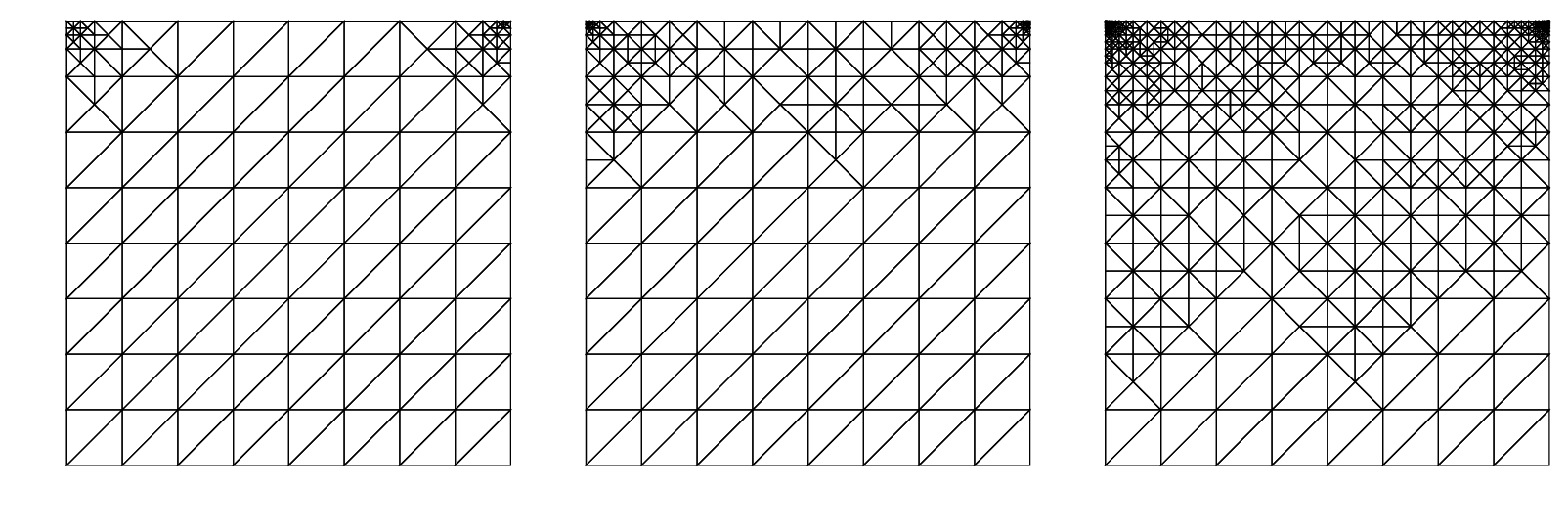} \\
	 \includegraphics[width=1\textwidth]{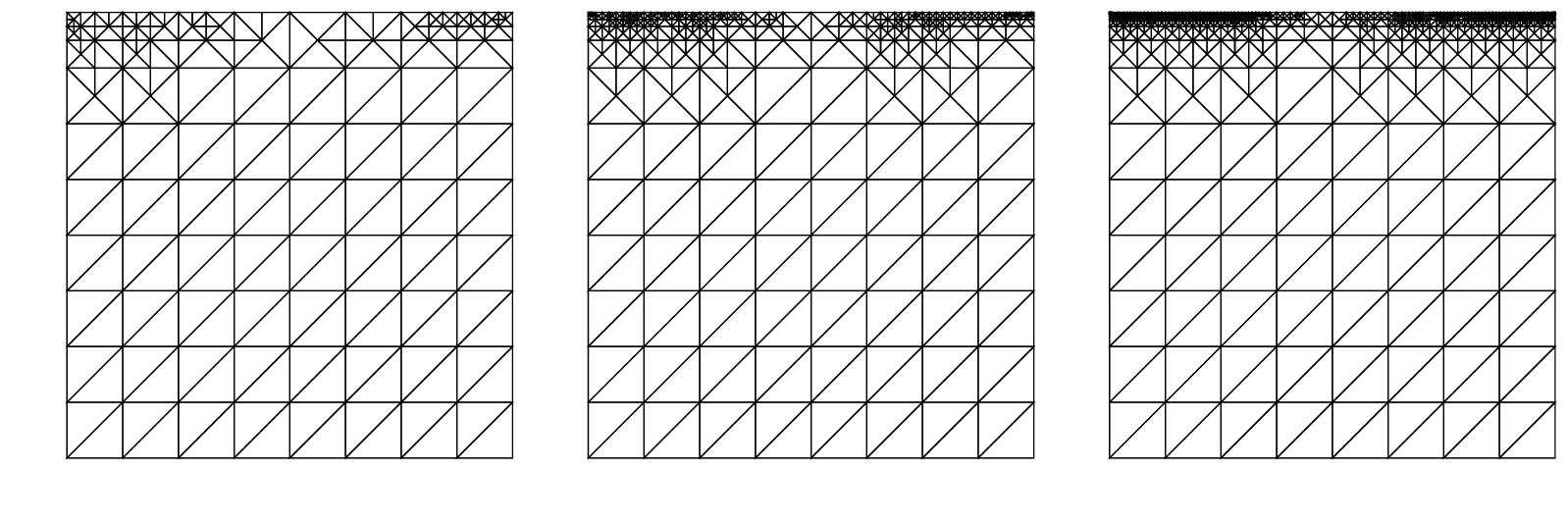}
	\caption{Example~\ref{ex:rand_diff}: Process of adaptively refined triangulations obtained by Algorithm~\ref{alg:adap} with the marking parameters 
		$\theta_h=0.5$ and  $\theta_q =0.5$ for the viscosity parameter $\nu=10^0$ (top)  and $\nu=10^{-2}$ (bottom).} \label{Ex1_mesh}	
\end{figure}

\begin{figure}[htp!]
	\centering
	\includegraphics[width=1\textwidth]{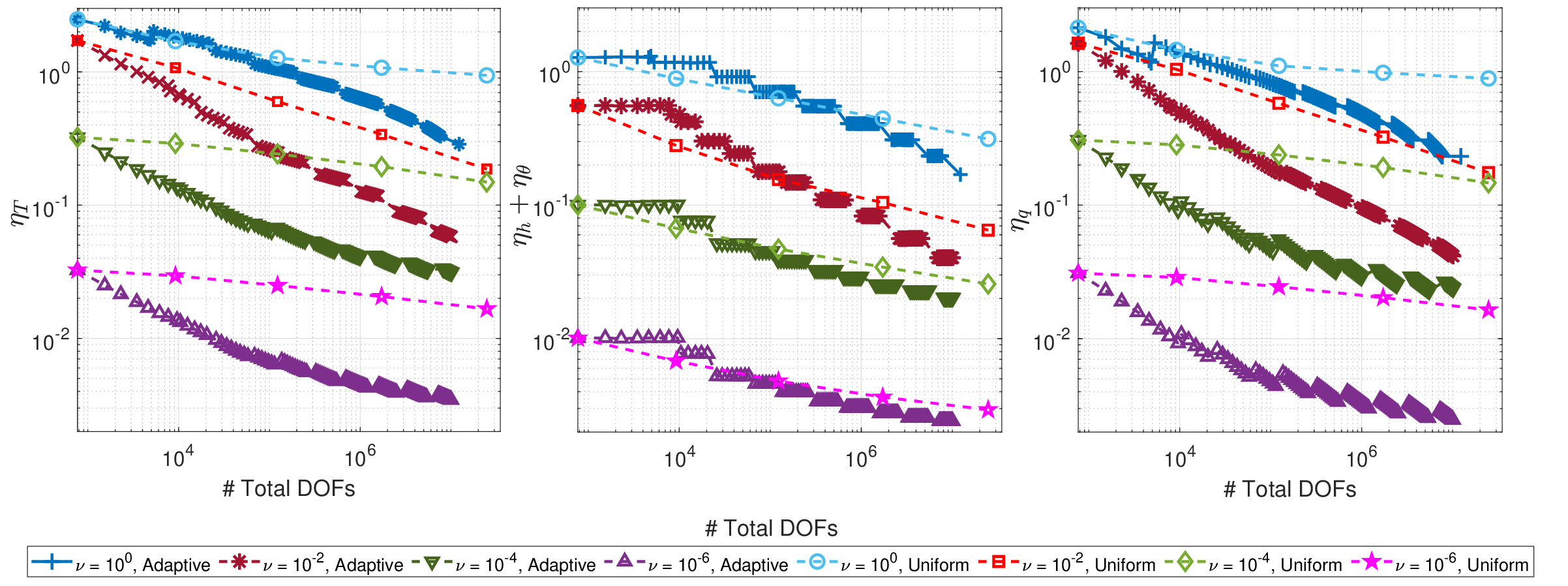}
	\caption{Example~\ref{ex:rand_diff}: Behaviour of error estimators  on the adaptively (with marking parameter $\theta_h = 0.5$, $\theta_q = 0.5$) and 
		uniformly generated spatial/parametric spaces  for different values of viscosity parameter $\nu$.} \label{Ex1_adapVSuniform_estimator}	
\end{figure}
\begin{figure}[htp!]
	\centering
	\includegraphics[width=1\textwidth]{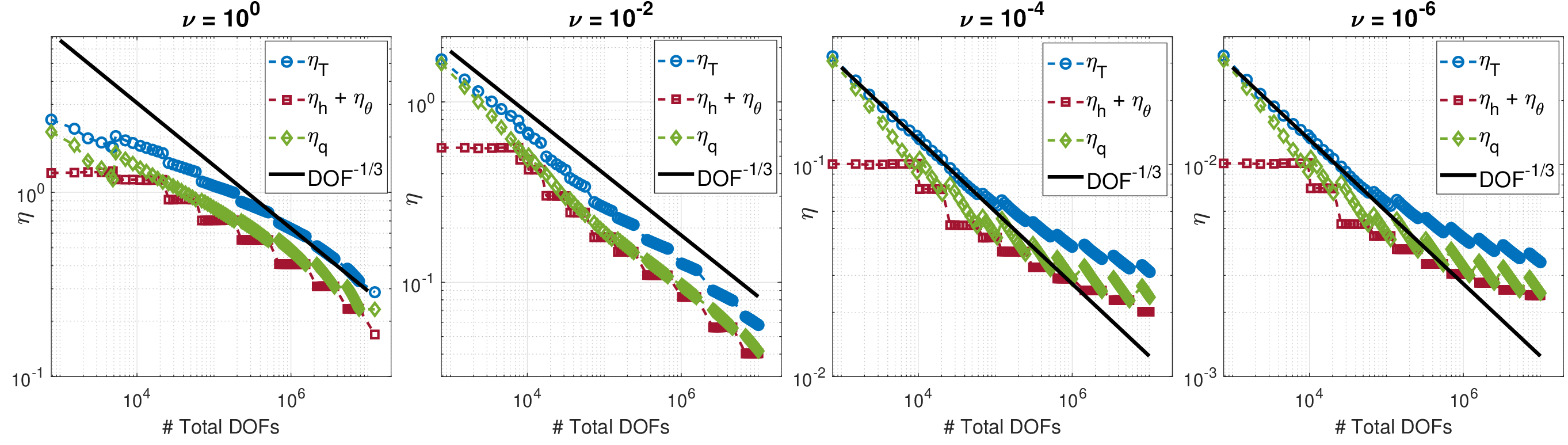}
	\caption{Example~\ref{ex:rand_diff}: Behaviour of error estimators on the adaptively generated spatial/parametric spaces with marking parameter 
		$\theta_h = 0.5$, $\theta_q = 0.5$  for different values of viscosity parameter $\nu$.}\label{Ex1_adap_estimator}	
\end{figure}

Figure~\ref{Ex1_adapVSuniform_estimator} displays the behavior of estimator $\eta_T$ and its spatial and parametric contributions 
$\eta_h + \eta_{\theta}$ and $\eta_q$, respectively, on the adaptively (with  the marking parameters $\theta_h=0.5$, $\theta_q =0.5$) and 
uniformly generated spatial/parametric spaces. 
Results on  Figure~\ref{Ex1_adapVSuniform_estimator} show that the total estimators on the adaptively generated spaces are superior than the ones on the uniformly generated spaces. From  Figure~\ref{Ex1_adap_estimator}, we also observe that the total error estimates decay with an overall rate of about $\mathcal{O}(DOF^{-1/3})$ even for the smaller values of $\nu$; see  \cite{MEigel_CJGittelson_CSchwab_EZander_2015, ABespalov_DPraetorius_LRocchi_MGuggeri_2019b} for the convergence analysis of adaptive stochastic Galerkin applied to elliptic problems. Next, the effect of marking parameters $\theta_q$ and $\theta_h$ is investigated in Tables~\ref{tab:Ex1_q} and~\ref{tab:Ex1_h}, respectively, for the viscosity parameter  $\nu=10^{-2}$. As expected, bigger values for the parameter $\theta_q$ result in more enrichment  in one loop, whereas smaller $\theta_q$  yields more optimal index  but more enrichment loops. It is also observed that the overall minimum total cost $N_{\text{cost}}$ is achieved for $\theta_q=0.8$ and $\theta_h=0.5$; but this configuration yields the lowest accuracy. When $\theta_q$ is fixed, although the iteration number decreases as $\theta_h$ increases, the process does not perform optimally; see also Figure~\ref{Ex1_theta}.  However, the adaptive algorithm converges  with an overall rate of about $\mathcal{O}(DOF^{-1/3})$ regardless of the marking parameters. Then, we investigate the effect of the correlation length $\ell$ and the standard deviation $\kappa$ on the estimator $\eta_T$ with the adaptively generated spatial/parametric spaces in Figure~\ref{Ex1_corre_kappa}. As expected, the performance of the estimator becomes worse when we decrease the value of correlation length $\ell$ and increase the value of the standard deviation $\kappa$. For instance, for $\kappa=2$, a spike happens due to probably a shortage of index set or mesh refinement. But as the degree of freedom increases, its convergence behavior  becomes better. Last, Figure~\ref{Ex1_eff}	displays the behaviour of the effectivity indices for various values of $\nu$. Even for the smaller values of $\nu$, the effectivity indices tends to a constant as iterations progress. 


\begin{table}[htp!]
	\tiny
	\centering
	\caption{Example~\ref{ex:rand_diff}: Results of  adaptive procedure with the viscosity parameter  $\nu=10^{-2}$  for  varying marking parameter $\theta_q$.}
	\label{tab:Ex1_q}
	\scalebox{0.9}{
		\begin{tabular}{ccc|cc|cc}
			$\theta_h = 0.5$                    &          & $\theta_q = 0.3$                                                                      &          & $\theta_q = 0.5$                                                                                      &          & $\theta_q = 0.8$                                                                                                      \\ \hline
			\multicolumn{1}{c|}{\# iter}                    &  & 105              &  & 99               &  & 82               \\ \hline
			\multicolumn{1}{c|}{\# Total DOFs}              &  & 10,043,280       &  & 10,109,880       &  & 10,481,064       \\ \hline
			\multicolumn{1}{c|}{\# $\mathfrak{B} $}         &  & 1392             &  & 2070             &  & 2146             \\ \hline
			\multicolumn{1}{c|}{$N_d$}                      &  & 7215             &  & 4884             &  & 4884             \\ \hline
			\multicolumn{1}{c|}{$N_{cost}$} &  & 185,422,335	  &  & 239,767,206  	&  & 156,010,791 \\ \hline	
			\multicolumn{1}{c|}{$\eta_T$}                   &  & 4.5662e-02       &  & 5.7960e-02       &  & 6.4655e-02       \\
			\multicolumn{1}{c|}{$\eta_h + \eta_{\theta}$} &  & 2.9286e-02       &  & 4.0376e-02       &  & 4.0335e-02       \\
			\multicolumn{1}{c|}{$\eta_q$} &  & 3.5034e-02       &  & 4.1583e-02       &  & 5.0531e-02  \\ \hline
			\multicolumn{1}{c|}{$\mathfrak{B}$} & iter = 1 & \begin{tabular}[c]{@{}c@{}}(0 0)\\ (1 0)\end{tabular}                                 & iter = 1 & \begin{tabular}[c]{@{}c@{}}(0 0)\\ (1 0)\end{tabular}                                                 & iter = 1 & \begin{tabular}[c]{@{}c@{}}(0 0)\\ (1 0)\end{tabular}                                                                 \\
			\multicolumn{1}{c|}{}               & iter = 2 & (0 1)                                                                                 & iter = 2 & \begin{tabular}[c]{@{}c@{}}(0 1)\\ (1 1)\end{tabular}                                                 & iter = 2 & \begin{tabular}[c]{@{}c@{}}(0 1)\\ (1 1)\end{tabular}                                                                 \\
			\multicolumn{1}{c|}{}               & iter = 3 & (0 0 1)                                                                               & iter = 3 & \begin{tabular}[c]{@{}c@{}}(0 0 1)\\ (0 1 1)\end{tabular}                                             & iter = 3 & \begin{tabular}[c]{@{}c@{}}(0 0 1)\\ (0 1 1)\\ (0 2 0)\end{tabular}                                                   \\
			\multicolumn{1}{c|}{}               & \multicolumn{2}{c|}{\vdots}          & \multicolumn{2}{c|}{\vdots}                                   & \multicolumn{2}{c}{\vdots}     \\
			\multicolumn{1}{l|}{}               & iter = 6 & \begin{tabular}[c]{@{}c@{}}(0 0 0 0 0 1)\\ (0 0 0 0 1 1)\\ (0 0 0 0 2 0)\end{tabular} & iter = 6 & \begin{tabular}[c]{@{}c@{}}(0 0 0 0 0 1)\\ (0 0 0 0 1 1)\\ (0 0 0 0 2 0)\\ (0 0 0 1 0 1)\end{tabular} & iter = 6 & \begin{tabular}[c]{@{}c@{}}(0 0 0 0 0 1)\\ (0 0 0 0 1 1)\\ (0 0 0 0 2 0)\\ (0 0 0 1 0 1)\\ (0 0 0 1 1 1)\end{tabular} \\
			
			\multicolumn{1}{c|}{}               & \multicolumn{2}{c|}{\vdots}          & \multicolumn{2}{c|}{\vdots}                                   & \multicolumn{2}{c}{\vdots}     \\
			\multicolumn{1}{l|}{}               & iter = 11 & \begin{tabular}[c]{@{}c@{}}(0 0 0 0 0 0 0 0 1)\\ (0 0 0 0 0 0 0 1 1)\\(0 0 0 0 0 0 0 2 0)\\(0 0 0 0 0 0 1 0 1)\end{tabular} 
			& iter = 11 & \begin{tabular}[c]{@{}c@{}}(0 0 0 0 0 0 0 0 1)\\ (0 0 0 0 0 0 0 1 1)\\(0 0 0 0 0 0 0 2 0)\\(0 0 0 0 0 0 1 0 1)\\(0 0 0 0 0 0 1 1 1)\end{tabular} 
			& iter = 11 & \begin{tabular}[c]{@{}c@{}}(0 0 0 0 0 0 0 0 1)\\ (0 0 0 0 0 0 0 1 1)\\(0 0 0 0 0 0 0 2 0)\\(0 0 0 0 0 0 1 0 1)\\(0 0 0 0 0 0 1 1 1)\\(0 0 0 0 0 0 1 2 0)\\(0 0 0 0 0 0 2 1 0)\end{tabular} \\
			
			\multicolumn{1}{c|}{}               & \multicolumn{2}{c|}{\vdots}          & \multicolumn{2}{c|}{\vdots}                                   & \multicolumn{2}{c}{\vdots}     \\
		\end{tabular}
	}
\end{table}

\begin{table}[htp!]
	\scriptsize
	\centering
	\caption{Example~\ref{ex:rand_diff}: Results of  adaptive procedure with the viscosity parameter  $\nu=10^{-2}$  for  varying marking parameter $\theta_h$.}
	\label{tab:Ex1_h}
	\begin{tabular}{cccc}
		$\theta_q = 0.5$                                & $\theta_h = 0.3$ & $\theta_h = 0.5$ & $\theta_h = 0.8$ \\ \hline
		\multicolumn{1}{c|}{\# iter}                    & 102              & 99               & 80               \\ \hline
		\multicolumn{1}{c|}{\# Total DOFs}              & 10,171,980       & 10,109,880       & 10,110,888       \\ \hline
		\multicolumn{1}{c|}{\# $\mathfrak{B} $}         & 2070             & 2070             & 1444             \\ \hline
		\multicolumn{1}{c|}{$N_d$}                      & 4914             & 4884             & 7002             \\ \hline
		\multicolumn{1}{c|}{$N_{cost}$} &   234,145,230	  & 239,767,206 	&   176,805,975	\\ \hline	
		\multicolumn{1}{c|}{$\eta_T$}                   & 5.7915e-02       & 5.7960e-02       & 5.2630e-02       \\
		\multicolumn{1}{c|}{$\eta_h + \eta_{\theta}$} & 4.0556e-02       & 4.0376e-02       & 3.4721e-02       \\
		\multicolumn{1}{c|}{$\eta_q$} & 4.1344e-02       & 4.1583e-02       & 3.9552e-02      
	\end{tabular}
\end{table}

\begin{figure}[htp!]
	\centering
	\includegraphics[width=0.70\textwidth]{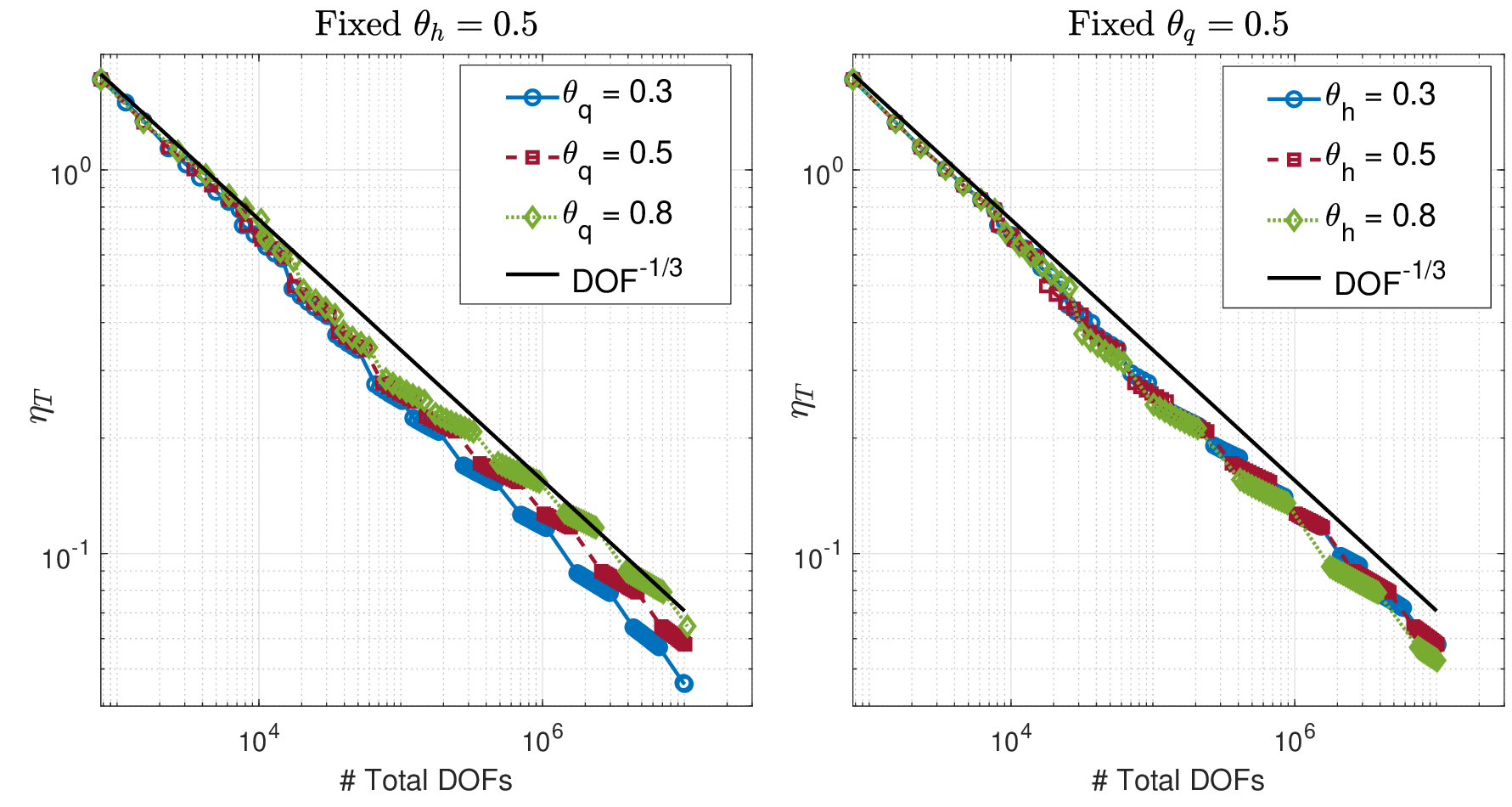}
	\caption{Example~\ref{ex:rand_diff}: Effect of the marking parameters $\theta_q$ (left) and $\theta_h$ (right) on the behaviour of estimator $\eta_T$  for the viscosity parameter $\nu=10^{-2}$.}\label{Ex1_theta}		
\end{figure}

\begin{figure}[htp!]
	\centering
	\includegraphics[width=0.70\textwidth]{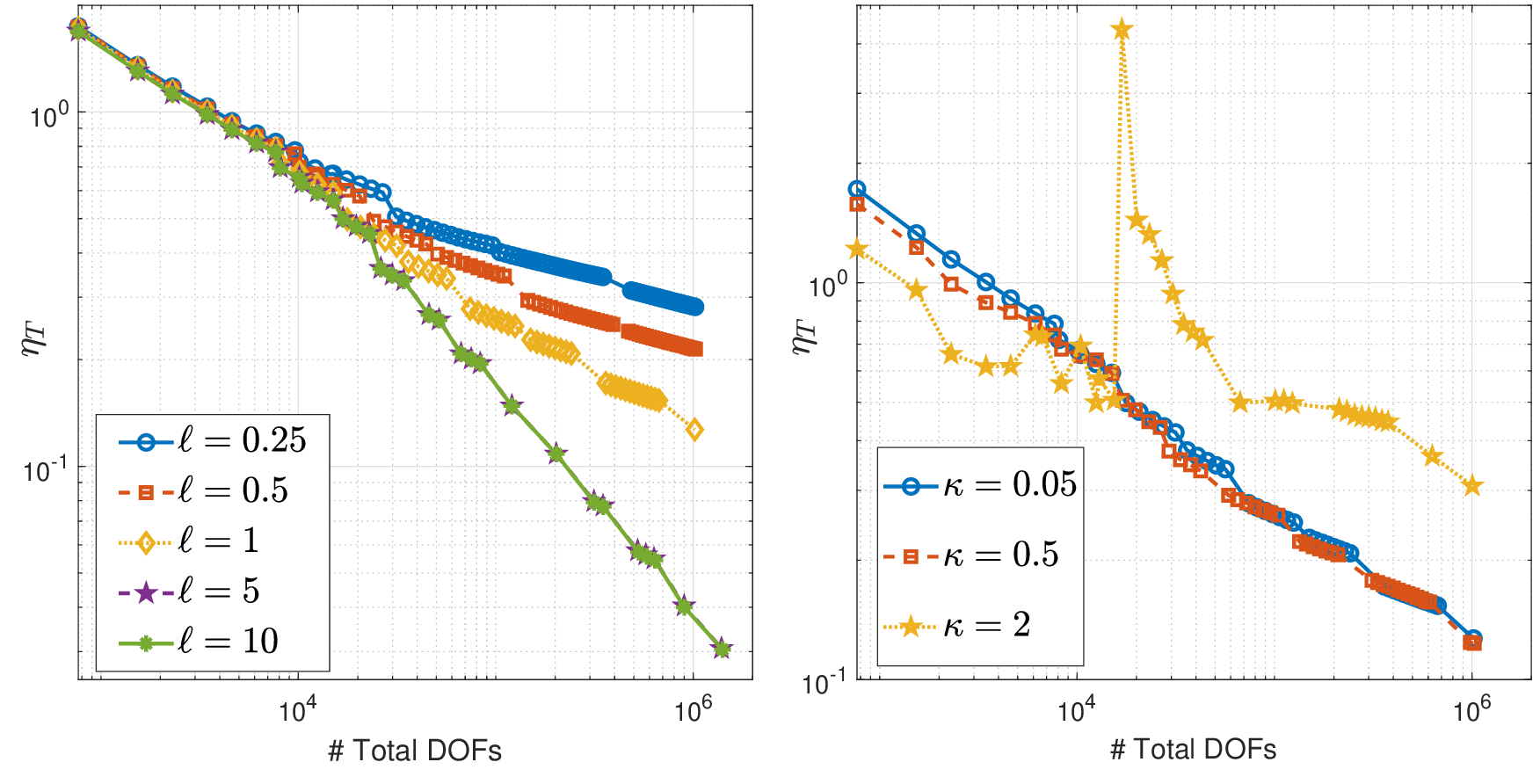}
	\caption{Example~\ref{ex:rand_diff}: Effect of the correlation length $\ell$ (left) and the standard deviation $\kappa$ (right) on the behaviour of estimator $\eta_T$ with adaptively generated spatial/parametric spaces for the viscosity parameter $\nu = 10^{-2}$.}\label{Ex1_corre_kappa}		
\end{figure}

\begin{figure}[htp!]
	\centering
	\includegraphics[width=1\textwidth]{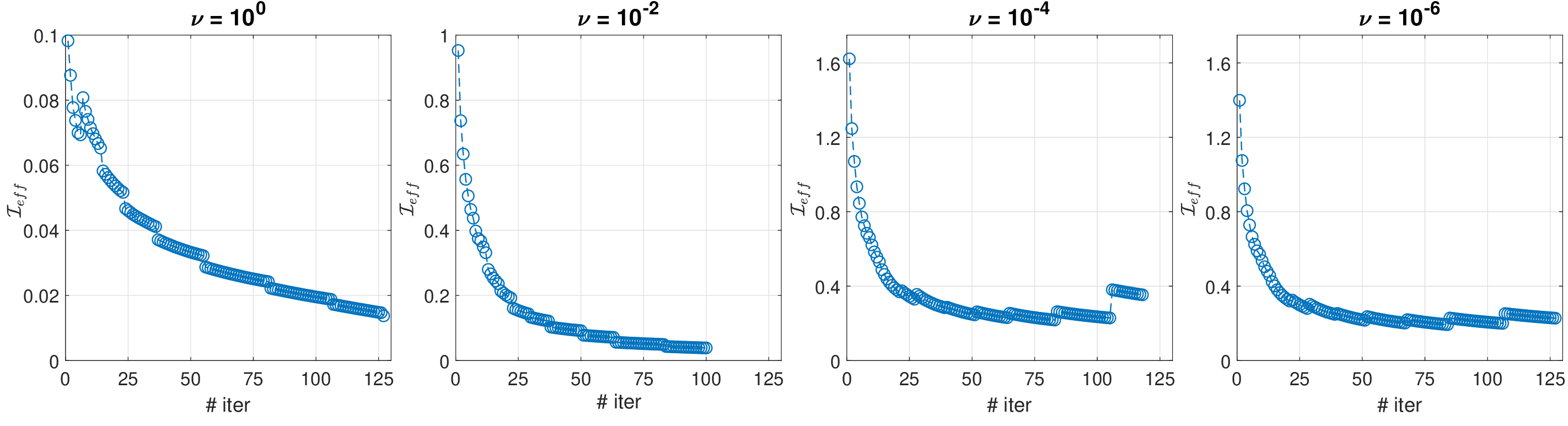}
	\caption{Example~\ref{ex:rand_diff}: Effectivity indices with $\theta_h = 0.5$, $\theta_q = 0.5$  for different values of the viscosity parameter $\nu$.}\label{Ex1_eff}		
\end{figure}

\subsubsection{Example with random convectivity}\label{ex:rand_conv}
Now, we consider  a convection diffusion equation with random convection coefficient in the spatial domain $\mathcal{D} = [-1,1]^2$ with the homogeneous Dirichlet boundary 
condition on $\partial \mathcal{D}$. To be precise, the problem  data is as follows
\[
a(\boldsymbol{x},\omega)=\nu > 0, \;  f(\boldsymbol{x})=0.5,\; \hbox{and} \;  \mathbf{b}(\boldsymbol{x},\omega)= \left(  z(\boldsymbol{x},\omega), z(\boldsymbol{x},\omega)\right)^T  \;  \hbox{with} \; \overline{z}(\boldsymbol{x}) = 1.
\]

Figure~\ref{Ex2_mean} displays the mean of the computed discrete solutions for various values of viscosity parameter $\nu$. Adaptively generated 
triangulations obtained by Algorithm~\ref{alg:adap} for the different values of viscosity $\nu$ are also displayed in Figure~\ref{Ex2_mesh}. As expected, 
most refinements occur around the boundaries $x_1 =1$ and $x_2=1$, where the solution exhibits boundary layers for the smaller values of $\nu$. Figure~\ref{Ex2_adapVSuniform_estimator}  shows that estimators exhibit a better convergence behavior for each value of viscosity parameter $\nu$ on the adaptively 
(with  the marking parameters $\theta_h=0.5$, $\theta_q =0.5$) than the ones on uniformly generated spatial/parametric spaces. In addition, the overall convergence 
rate of the total estimator $\eta_T$ is about  $\mathcal{O}(DOF^{-1/3})$ for the smaller values of the viscosity $\nu$, see Figure~\ref{Ex2_adap_estimator}, as the previous Example~\ref{ex:rand_diff}. 

\begin{figure}[htp!]
	\centering
	\includegraphics[width=1\textwidth]{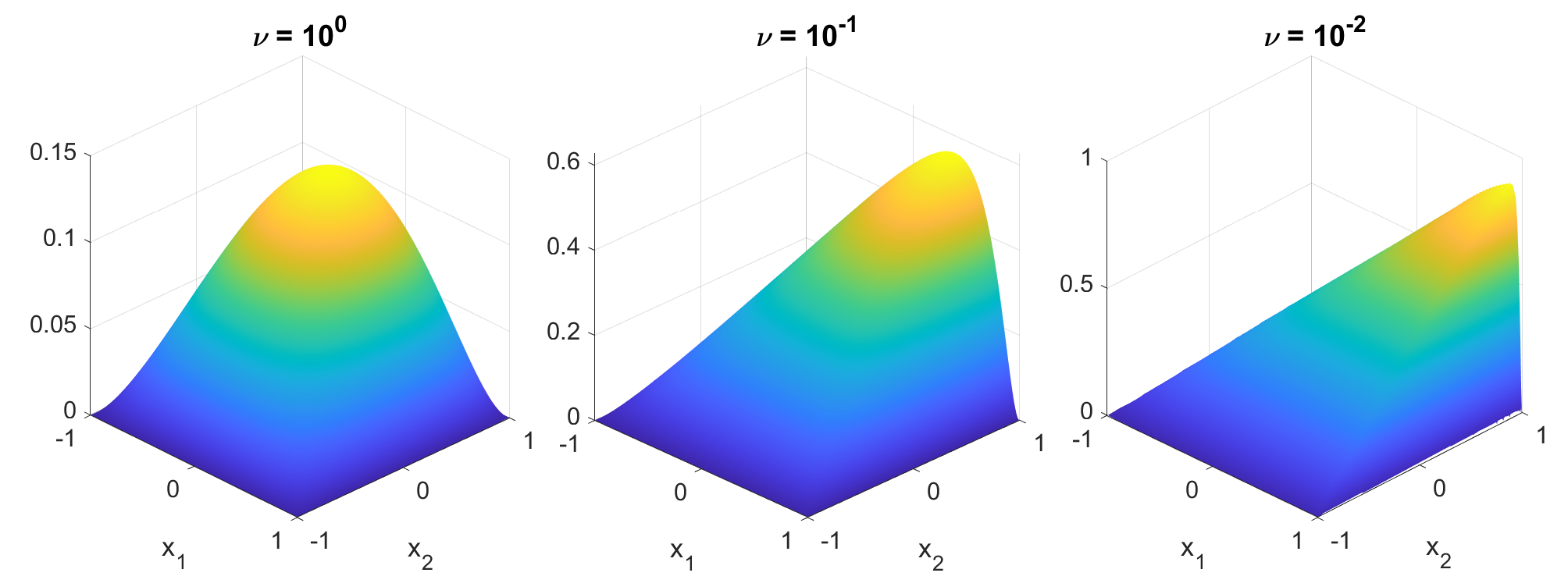}
	\caption{Example~\ref{ex:rand_conv}: Mean of SDG solutions for different values of viscosity parameter $\nu$.}
	\label{Ex2_mean}		
\end{figure}

\begin{figure}[htp!]
	\centering
	\includegraphics[width=1\textwidth]{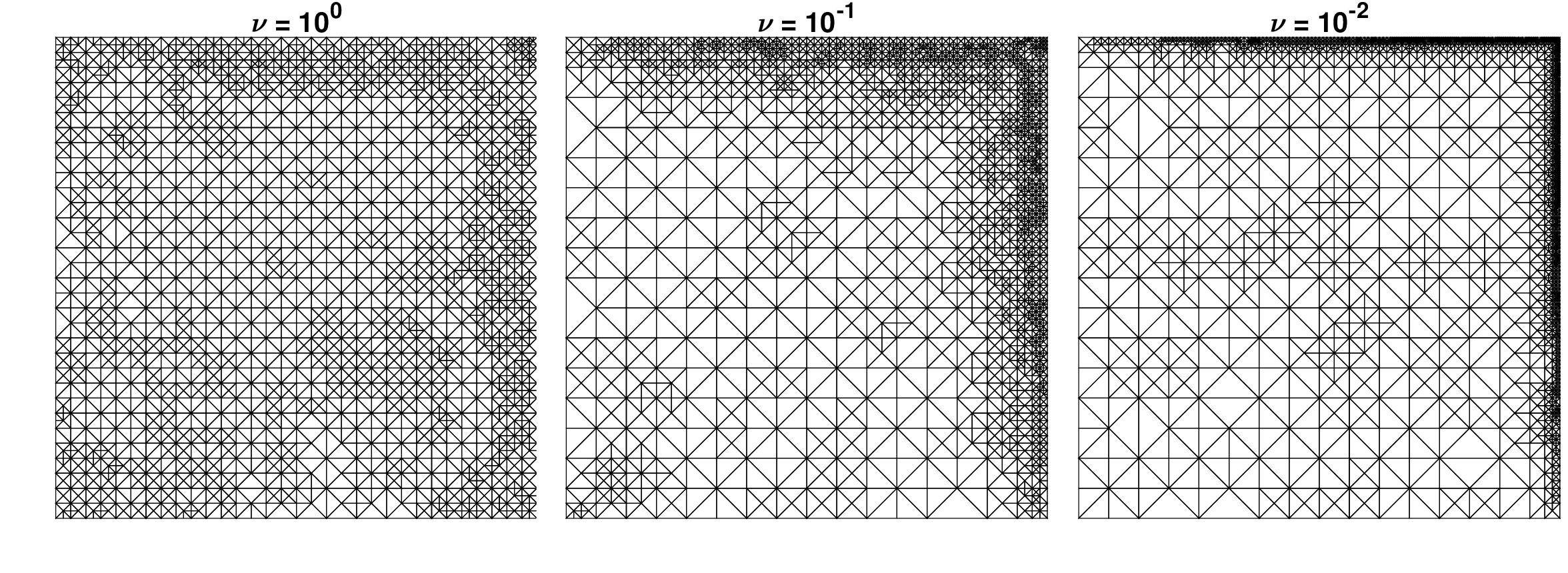}
	\caption{Example~\ref{ex:rand_conv}: Adaptively refined triangulations obtained by Algorithm~\ref{alg:adap} with the marking parameters $\theta_h=0.5$, $\theta_q =0.5$ for 
		the viscosity parameter $\nu=10^0 \hbox{ with } iter = 8, N_d = 10593$ (left), $\nu=10^{-1} \hbox{ with }  iter= 30, N_d = 11385$ (middle),  and $\nu=10^{-2} \hbox{ with }  iter = 65, N_d=13062$ (right).} \label{Ex2_mesh}	
\end{figure}

\begin{figure}[htp!]
	\centering
	\includegraphics[width=1\textwidth]{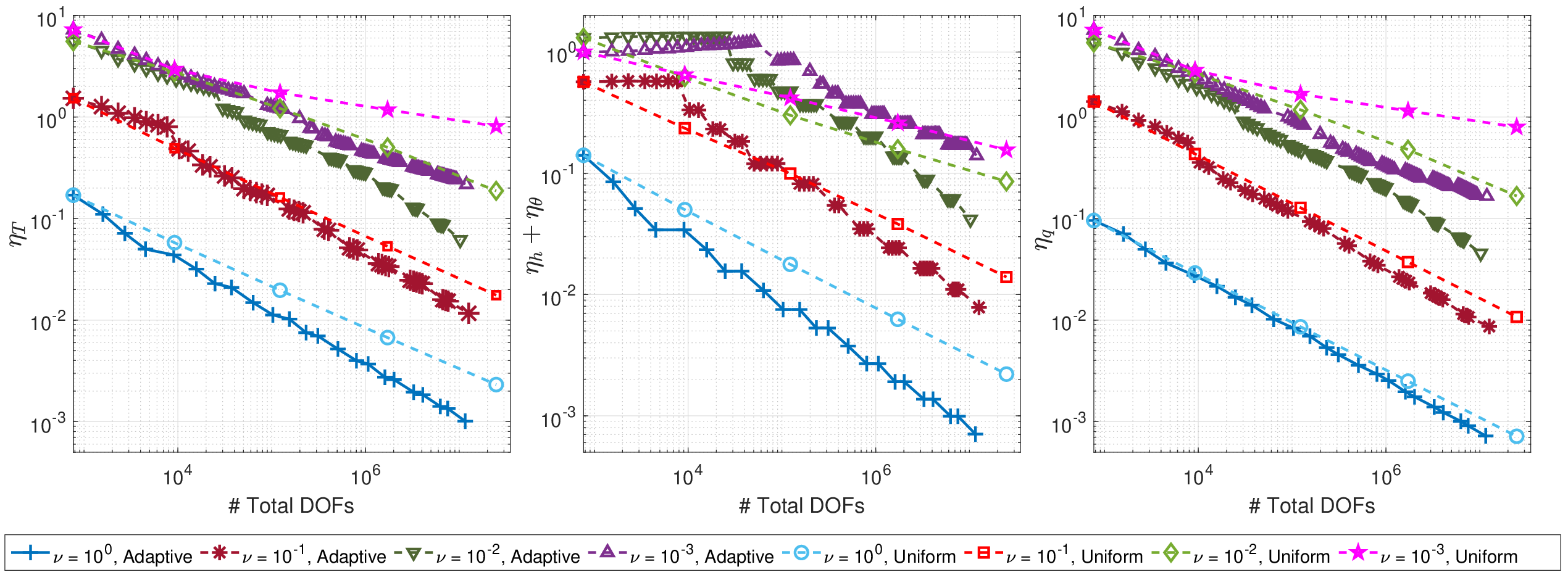}
	\caption{Example~\ref{ex:rand_conv}: Behaviour of error estimators  on the adaptively (with marking parameter $\theta_h = 0.5$, $\theta_q = 0.5$) and 
		uniformly generated spatial/parametric spaces  for different values of viscosity parameter $\nu$.} \label{Ex2_adapVSuniform_estimator}	
\end{figure}

\begin{figure}[htp!]
	\centering
	\includegraphics[width=1\textwidth]{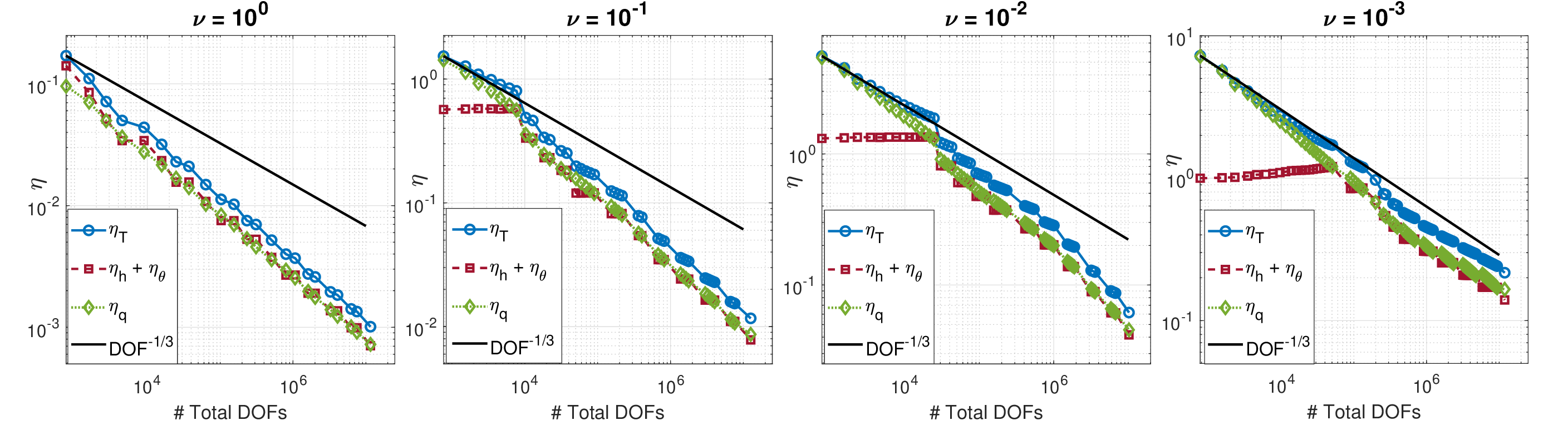}
	\caption{Example~\ref{ex:rand_conv}: Behaviour of error estimators on the adaptively generated spatial/parametric spaces with marking parameter $\theta_h = 0.5$, $\theta_q = 0.5$  
		for different values of viscosity parameter $\nu$.}\label{Ex2_adap_estimator}	
\end{figure}

\begin{table}[htp!]
	\tiny
	\centering
	\caption{Example~\ref{ex:rand_conv}: Results of adaptive procedure with the viscosity parameter  $\nu=10^{-2}$  for varying marking parameter  $\theta_q $.}
	\label{tab:Ex2_q}
	\scalebox{0.9}{
		\begin{tabular}{ccc|cc|cc}
			$\theta_h = 0.5$                            & & $\theta_q = 0.3$  & & $\theta_q = 0.5$   & & $\theta_q = 0.8$  \\ \hline
			\multicolumn{1}{c|}{\# iter}                & & 78   			  & & 64  				 & & 59 \\ \hline
			\multicolumn{1}{c|}{\# Total DOFs}          & & 12,880,665        & & 10,240,608 		 & & 12,735,090\\ \hline
			\multicolumn{1}{c|}{\# $\mathfrak{B} $}     & & 721               & & 784   			 & & 1010 \\ \hline
			\multicolumn{1}{c|}{$N_d$}                  & & 17865             & & 13062 			 & & 12609   \\ \hline
			\multicolumn{1}{c|}{$N_{cost}$}  & & 135,575,760	 & & 76,888,125 	& & 87,144,246 \\ \hline	
			\multicolumn{1}{c|}{$\eta_T$}               & & 4.7329e-02        & & 6.1645e-02         & & 6.4343e-02   \\
			\multicolumn{1}{c|}{$\eta_h + \eta_\theta$} & & 3.1194e-02        & & 4.1618e-02  & & 4.3911e-02   \\
			\multicolumn{1}{c|}{$\eta_q$}               & & 3.5594e-02        & & 4.5476e-02  & & 4.7030e-02\\ \hline
			\multicolumn{1}{c|}{$\mathfrak{B}$}         & iter = 1 & \begin{tabular}[c]{@{}c@{}}(0 0)\\  (1 0)\end{tabular}  & iter = 1 & \begin{tabular}[c]{@{}c@{}}(0 0)\\  (1 0)\end{tabular} & iter = 1 & \begin{tabular}[c]{@{}c@{}}(0 0)\\  (1 0)\end{tabular} \\
			\multicolumn{1}{c|}{}                       & iter = 2 & (0 1)  & iter = 2 & \begin{tabular}[c]{@{}c@{}}(0 1)\\  (1 1)\end{tabular} & iter = 2 & \begin{tabular}[c]{@{}c@{}}(0 1)\\  (1 1)\end{tabular}  \\
			\multicolumn{1}{c|}{}                       & iter = 3 & (0 0 1) & iter = 3 & \begin{tabular}[c]{@{}c@{}}(0 0 1)\\  (0 1 1)\end{tabular}  & iter = 3 & \begin{tabular}[c]{@{}c@{}}(0 0 1)\\  (0 1 1)\\  (0 2 0)\end{tabular} \\
			\multicolumn{1}{c|}{}                       &          & $\vdots$   &          & $\vdots$  &     & $\vdots$   \\
			\multicolumn{1}{c|}{}                       & iter = 6 & \begin{tabular}[c]{@{}c@{}}(0 0 0 0 0 1)\\ (0 0 0 0 1 1)\end{tabular}   & iter = 6 & \begin{tabular}[c]{@{}c@{}}(0 0 0 0 0 1)\\  (0 0 0 0 1 1)\\  (0 0 0 0 2 0)\\  (0 0 0 1 0 1)\end{tabular}  & iter = 6 & \begin{tabular}[c]{@{}c@{}}(0 0 0 0 0 1)\\  (0 0 0 0 1 1)\\  (0 0 0 0 2 0)\\  (0 0 0 1 0 1)\\  (0 0 0 1 1 1)\end{tabular}  \\
			\multicolumn{1}{c|}{}                       &          & $\vdots$   &          & $\vdots$  &     & $\vdots$ \\
			\multicolumn{1}{c|}{}                       & iter = 9 & \begin{tabular}[c]{@{}c@{}}( 0 0 0 0 0 0 0 0 1)\\ (0 0 0 0 0 0 0 1 1)\\ (0 0 0 0 0 0 0 2 0)\end{tabular} & iter = 9 & \begin{tabular}[c]{@{}c@{}}(0 0 0 0 0 0 0 0 1)\\ (0 0 0 0 0 0 0 1 1)\\ (0 0 0 0 0 0 0 2 0)\\ (0 0 0 0 0 0 1 0 1)\\ (0 0 0 0 0 0 1 1 1)\end{tabular} & iter = 9 & \begin{tabular}[c]{@{}c@{}}(0 0 0 0 0 0 0 0 1)\\ (0 0 0 0 0 0 0 1 1)\\ (0 0 0 0 0 0 0 2 0)\\ (0 0 0 0 0 0 1 0 1)\\ (0 0 0 0 0 0 1 1 1)\\ (0 0 0 0 0 0 1 2 0)\\ (0 0 0 0 0 0 2 0 1)\\ (0 0 0 0 0 0 2 1 0)\end{tabular} \\
			\multicolumn{1}{c|}{}                       &          & $\vdots$  &          & $\vdots$   &          & $\vdots$
		\end{tabular}
	}
\end{table}

\begin{table}[htp!]
	\scriptsize
	\centering
	\caption{Example~\ref{ex:rand_conv}: Results of  adaptive algorithm with the viscosity parameter  $\nu=10^{-2}$  for  varying marking parameter $\theta_h$.}
	\label{tab:Ex2_h}
	\begin{tabular}{cccc}
		$\theta_q = 0.5$                                & $\theta_h = 0.3$ & $\theta_h = 0.5$ & $\theta_h = 0.8$ \\ \hline
		\multicolumn{1}{c|}{\# iter}                    & 75               & 64               & 55               \\ \hline
		\multicolumn{1}{c|}{\# Total DOFs}              & 10,265,400       & 10,240,608       & 11,253,450       \\ \hline
		\multicolumn{1}{c|}{\# $\mathfrak{B} $}         & 900              & 784              & 650             \\ \hline
		\multicolumn{1}{c|}{$N_d$}                      & 11406            & 13062            & 17313             \\ \hline
		\multicolumn{1}{c|}{$N_{cost}$} &   118,714,650	  & 76,888,125 	&   65,584,812	\\ \hline
		\multicolumn{1}{c|}{$\eta_T$}                   & 6.6956e-02       & 6.1645e-02       & 5.2375e-02       \\
		\multicolumn{1}{c|}{$\eta_h + \eta_{\theta}$} & 4.6923e-02       & 4.1618e-02    	  & 3.5066e-02       \\
		\multicolumn{1}{c|}{$\eta_q$} & 4.7763e-02       & 4.5476e-02       & 3.8905e-02     
	\end{tabular}
\end{table}

\begin{figure}[htp!]
	\centering
	\includegraphics[width=0.70\textwidth]{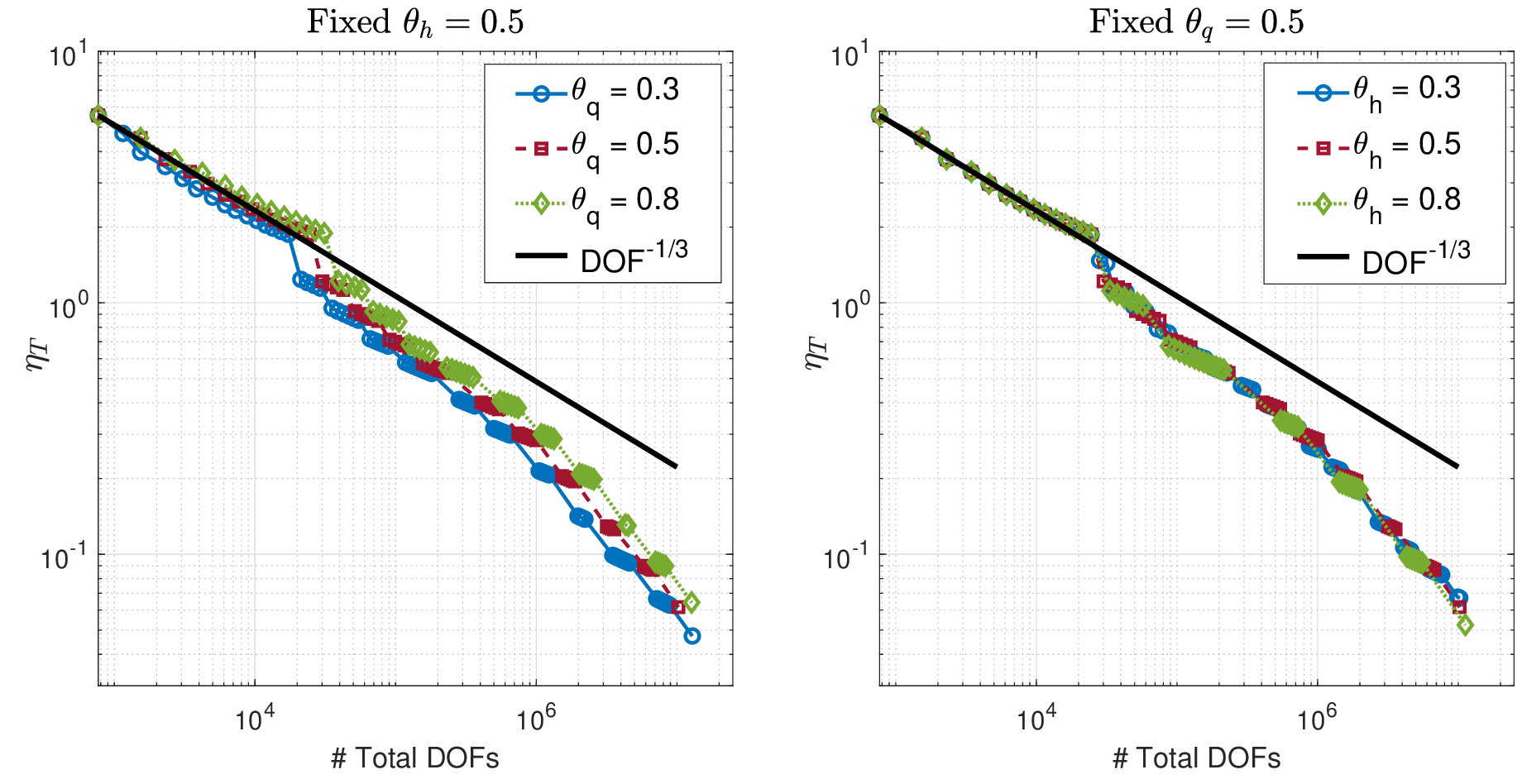}
	\caption{Example~\ref{ex:rand_conv}: Effect of the marking parameters $\theta_q$ (left) and $\theta_h$ (right) on the behaviour of estimator $\eta_T$  for the viscosity 
		parameter $\nu=10^{-2}$.}\label{Ex2_theta}		
\end{figure}

Next, Table~\ref{tab:Ex2_q} shows the results of adaptive algorithm with the viscosity parameter  $\nu=10^{-2}$  for varying marking parameter  $\theta_q $.  When the value of the marking parameter $\theta_q$ is increased, we observe that the number of iteration decreases, while the size of index set $\mathfrak{B}$ increases. However, as the previous example, we could not obtain an optimal process for the making parameter $\theta_h$; see, Table~\ref{tab:Ex2_h} and Figure~\ref{Ex2_theta}. Additionally, the highest accuracy is obtained when $\theta_h = 0.5$ and $\theta_q = 0.3$, but it requires the most computation cost. The performance of the estimator $\eta_T$ on adaptively generated spatial/parametric spaces for different values of  the correlation length $\ell$ and the standard deviation $\kappa$ is given in Figure~\ref{Ex2_corre_kappa}. Last, the behaviour of the effectivity indices  is displayed in Figure~\ref{Ex2_eff}.


\begin{figure}[htp!]
	\centering
	\includegraphics[width=0.70\textwidth]{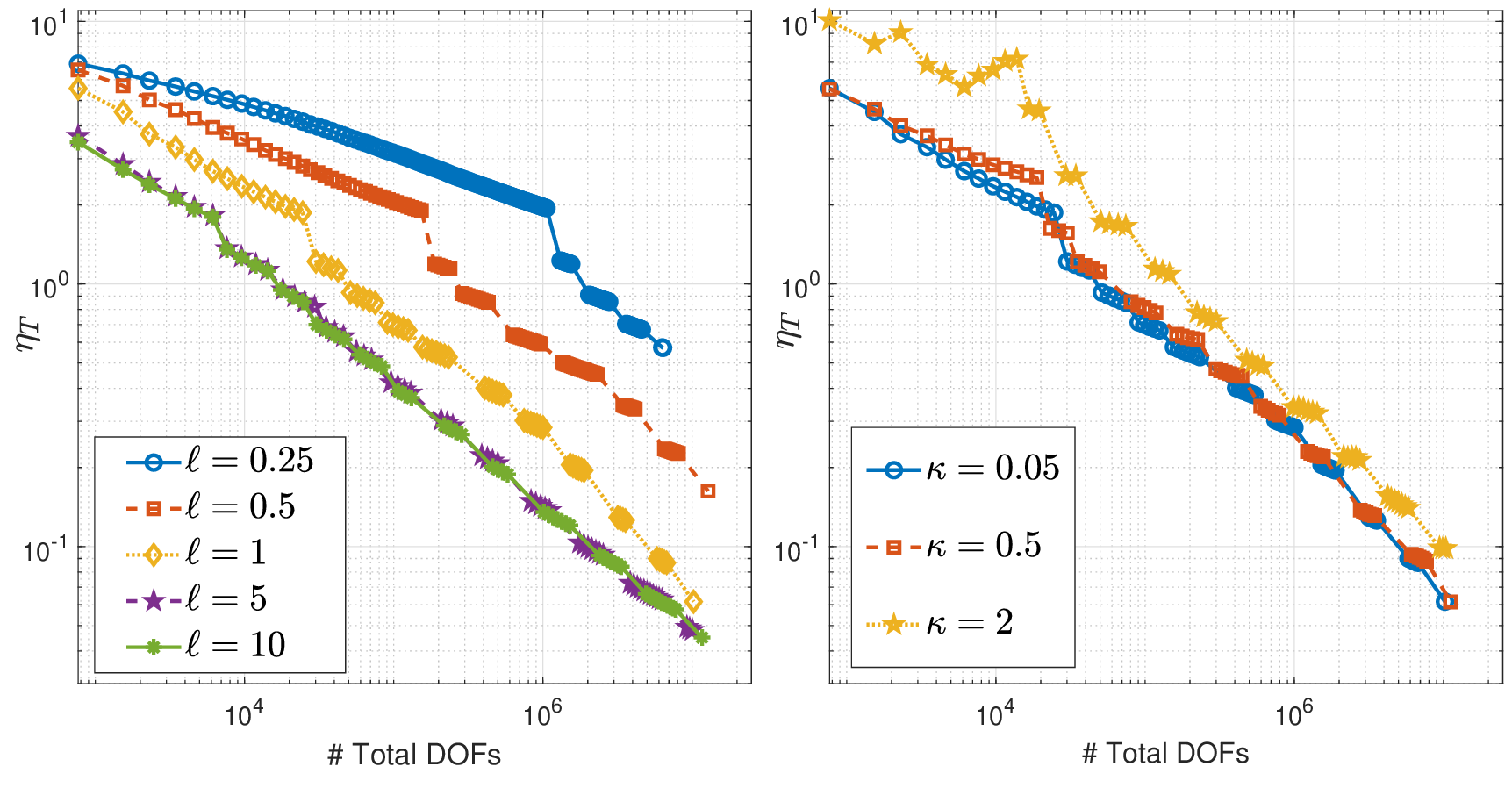}
	\caption{Example~\ref{ex:rand_conv}: Effect of the correlation length $\ell$ (left) and the standard deviation $\kappa$ (right) on the behaviour of estimator $\eta_T$ with 
		adaptively generated spatial/parametric spaces for the viscosity parameter $\nu = 10^{-2}$.}\label{Ex2_corre_kappa}		
\end{figure}

\begin{figure}[htp!]
	\centering
	\includegraphics[width=1\textwidth]{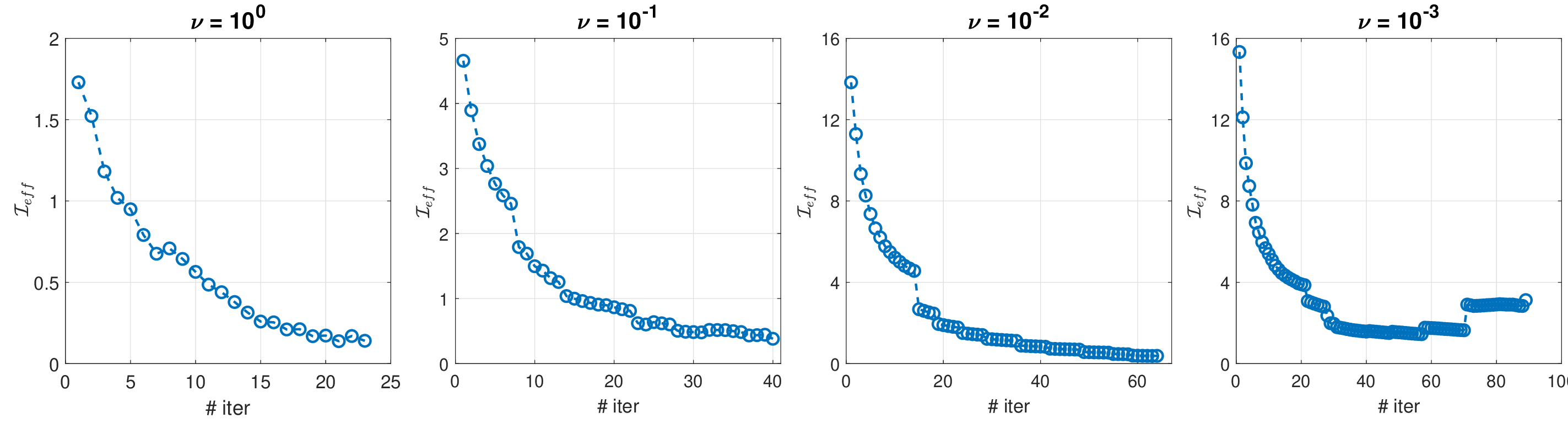}
	\caption{Example~\ref{ex:rand_conv}: Effectivity indices with $\theta_h = 0.5$, $\theta_q = 0.5$  for different values of the viscosity parameter $\nu$.}\label{Ex2_eff}		
\end{figure}


\begin{figure}[htp!]
	\centering
	\includegraphics[width=1\textwidth]{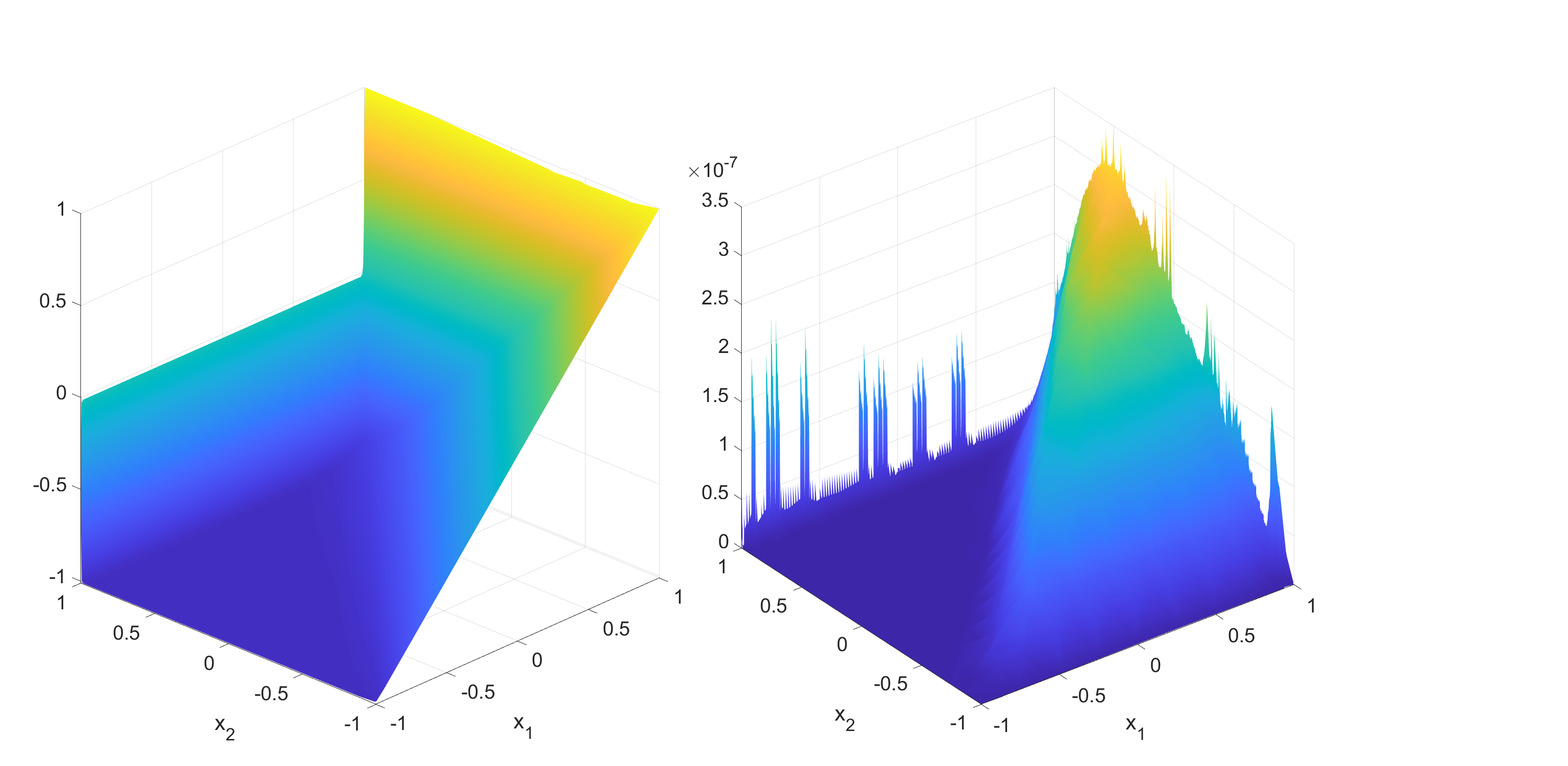}
	\caption{Example~\ref{ex:rand_both}: Mean (left) and variance (right) of SDG solutions obtained from Algorithm \ref{alg:adap} with the marking parameters $\theta_h=0.5$, $\theta_q =0.5$.}\label{Ex3_mean}		
\end{figure}

\begin{figure}[htp!]
	\centering
	\includegraphics[width=1\textwidth]{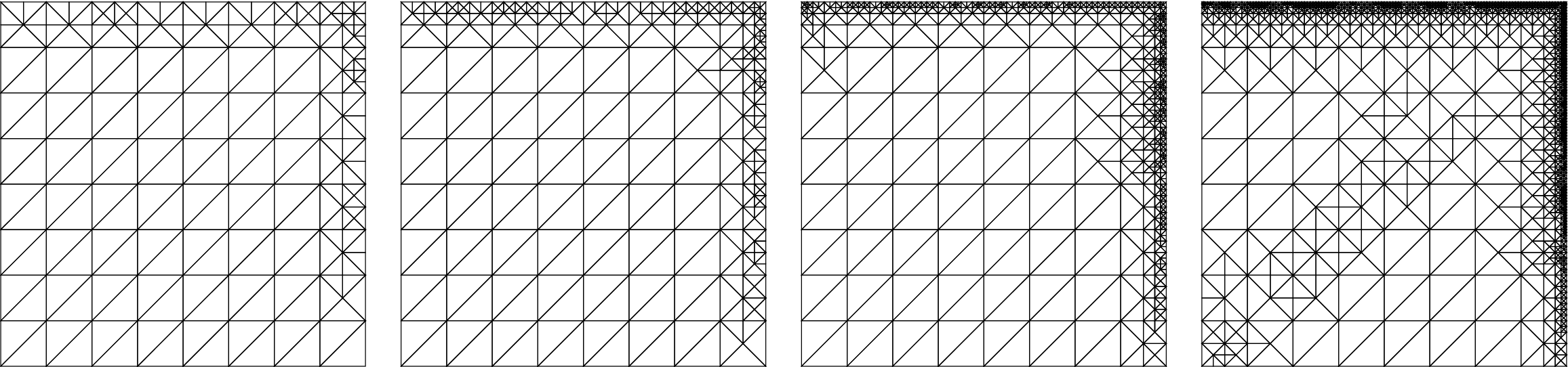} 
	\caption{Example~\ref{ex:rand_both}: Process of adaptively refined triangulations obtained by Algorithm~\ref{alg:adap} with the marking parameters $\theta_h=0.5$, $\theta_q =0.5$.} \label{Ex3_mesh}	
\end{figure}

\subsubsection{Example with random diffusivity and convectivity}\label{ex:rand_both}
Next, we consider our model problem \eqref{eqn:m1} with both random diffusion and convection coefficients on the spatial domain $\mathcal{D} = [-1,1]^2$ with the nonhomogeneous Dirichlet boundary conditions given in \eqref{nonDBC}. Rest of the data  is as follows
\[
a(\boldsymbol{x},\omega)=\nu   z(\boldsymbol{x},\omega), \quad \mathbf{b}(\boldsymbol{x},\omega)= \left(  z(\boldsymbol{x},\omega), z(\boldsymbol{x},\omega)\right)^T,  \quad \hbox{and} \quad  f(\boldsymbol{x})=0,
\]
where $\nu = 10^{-2}$ and $\overline{z}(\boldsymbol{x}) = 1$, i.e., mean of underlying random field. We note that the diffusion and convection terms are driven by the same random field. 

Figure~\ref{Ex3_mean} displays mean and variance of the SDG solutions obtained from Algorithm \ref{alg:adap} with the marking parameters 
$\theta_h=0.5$, $\theta_q =0.5$. From Figure~\ref{Ex3_mean} and the direction of velocity, we expect more refinement around the boundaries 
$x_1 = 1$, $x_2=1$, and the line in the direction of $(1,1)^T$; see Figure~\ref{Ex3_mesh} for the process of adaptively refined triangulations. 
Figure~\ref{Ex3_estimator}	 compares  the error estimators on the adaptively and uniformly generated spatial/parametric spaces and gives the progress of the effectivity indices as the iteration progress.

\begin{figure}[htp!]
	\centering
	\includegraphics[width=1\textwidth]{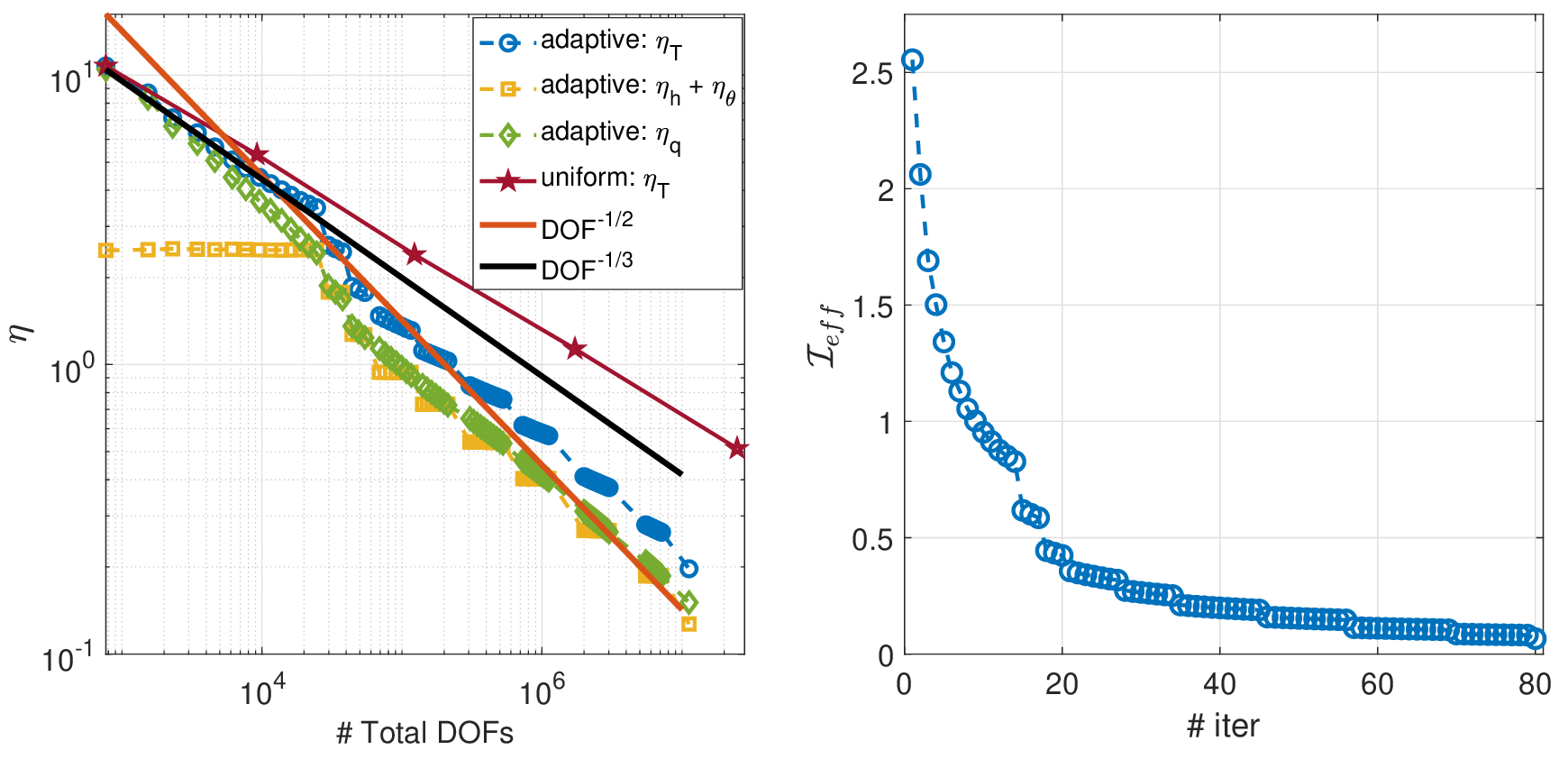}
	\caption{Example \ref{ex:rand_both}: Behaviour of error estimators (left) on the adaptively and uniformly generated spatial/parametric spaces and of effectivity indices with the marking parameters $\theta_h=0.5$, $\theta_q =0.5$. }
	\label{Ex3_estimator}		
\end{figure}


\subsubsection{Example with jump coefficient}\label{ex:rand_jump}
Last, we test our estimator $\eta_T$  \eqref{totalerror} on the random (jump) discontinuous coefficient.  Data of the problem is given by 
\[
\mathcal{D} = [-1,1]^2, \quad f(\boldsymbol{x})=0, \quad \hbox{and} \quad \mathbf{b}(\boldsymbol{x})=(0,1)^T
\]
with the nonhomogeneous Dirichlet boundary condition \eqref{nonDBC}. Stochastic jump diffusion coefficient  is  $a(\boldsymbol{x}, \omega) = \nu z(\boldsymbol{x}, \omega)$, where 
$\nu=10^{-2}$ is the viscosity parameter and $z(\boldsymbol{x}, \omega)$ is  a random variable having the mean as follows
\begin{align}
	\bar{z}(\boldsymbol{x}) = \begin{cases}
		10^4, & \hbox{if } \boldsymbol{x} \in \mathcal{D}_1,\\
		1,    & \hbox{if } \boldsymbol{x} \in \mathcal{D}_2,
	\end{cases}
\end{align}
where $\mathcal{D} = \mathcal{D}_1 \cup \mathcal{D}_2$. To observe the sensitivity of the adaptive algorithm with respect to the decomposition of the domain, we test the 
example with different choices, see also Figure~\ref{Ex4_domain}, as 
\begin{itemize}
	\item \emph{Half partition}: $\mathcal{D}_1 = [-1, 0] \times [-1,1]$ and $\mathcal{D}_2 = (0, 1] \times [-1,1]$,
	\item \emph{Rectangle partition}: $ \mathcal{D}_1 = \big( [-3/4,-1/3] \times [-3/4,-1/3] \big) \bigcup \big([1/3, 3/4] \times [-3/4, -1/3] \big) $ and 
	$\mathcal{D}_2 = \mathcal{D} \backslash \mathcal{D}_1$.
\end{itemize}

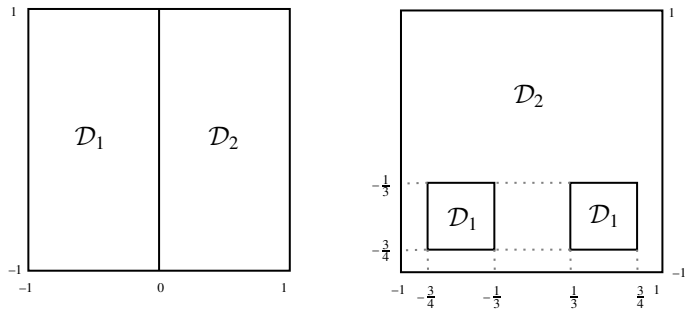
\begin{figure}[htp!]
	\centering
	\tikzset{every picture/.style={line width=0.75pt}} 
	\begin{tikzpicture}[x=0.75pt,y=0.75pt,yscale=-1,xscale=1,scale=0.75]
		
		\draw   (52,22) -- (226,22) -- (226,196) -- (52,196) -- cycle ;
		\draw    (139,22) -- (139,196) ;
		\draw   (300,23) -- (474,23) -- (474,197) -- (300,197) -- cycle ;
		\draw   (317.71,137.57) -- (362.14,137.57) -- (362.14,182) -- (317.71,182) -- cycle ;
		\draw   (412.71,137.57) -- (457.14,137.57) -- (457.14,182) -- (412.71,182) -- cycle ;
		\draw [color={rgb, 255:red, 128; green, 128; blue, 128 }  ,draw opacity=1 ] [dash pattern={on 0.84pt off 2.51pt}]  (317.71,182) -- (318,196.5) ;
		\draw [color={rgb, 255:red, 128; green, 128; blue, 128 }  ,draw opacity=1 ] [dash pattern={on 0.84pt off 2.51pt}]  (362.14,182) -- (362.43,196.5) ;
		\draw [color={rgb, 255:red, 128; green, 128; blue, 128 }  ,draw opacity=1 ] [dash pattern={on 0.84pt off 2.51pt}]  (412.71,182) -- (413,196.5) ;
		\draw [color={rgb, 255:red, 128; green, 128; blue, 128 }  ,draw opacity=1 ] [dash pattern={on 0.84pt off 2.51pt}]  (457.14,182) -- (457.43,196.5) ;
		\draw [color={rgb, 255:red, 128; green, 128; blue, 128 }  ,draw opacity=1 ] [dash pattern={on 0.84pt off 2.51pt}]  (316.71,182) -- (300,182.5) ;
		\draw [color={rgb, 255:red, 128; green, 128; blue, 128 }  ,draw opacity=1 ] [dash pattern={on 0.84pt off 2.51pt}]  (317.71,137.57) -- (301,138.07) ;
		\draw [color={rgb, 255:red, 128; green, 128; blue, 128 }  ,draw opacity=1 ] [dash pattern={on 0.84pt off 2.51pt}]  (412.71,182) -- (362.14,182) ;
		\draw [color={rgb, 255:red, 128; green, 128; blue, 128 }  ,draw opacity=1 ] [dash pattern={on 0.84pt off 2.51pt}]  (412.71,137.57) -- (362.14,137.57) ;

		\draw (44,202.4) node [anchor=north west][inner sep=0.75pt]  [font=\tiny]  {$-1$};
		\draw (136,202.4) node [anchor=north west][inner sep=0.75pt]  [font=\tiny]  {$0$};
		\draw (218,202.4) node [anchor=north west][inner sep=0.75pt]  [font=\tiny]  {$1$};
		\draw (38,18.4) node [anchor=north west][inner sep=0.75pt]  [font=\tiny]  {$1$};
		\draw (37,190.4) node [anchor=north west][inner sep=0.75pt]  [font=\tiny]  {$-1$};
		\draw (81,98.4) node [anchor=north west][inner sep=0.75pt]    {$\mathcal{D}_{1}$};
		\draw (170,98.4) node [anchor=north west][inner sep=0.75pt]    {$\mathcal{D}_{2}$};
		\draw (292,203.4) node [anchor=north west][inner sep=0.75pt]  [font=\tiny]  {$-1$};
		\draw (466,203.4) node [anchor=north west][inner sep=0.75pt]  [font=\tiny]  {$1$};
		\draw (476,19.4) node [anchor=north west][inner sep=0.75pt]  [font=\tiny]  {$1$};
		\draw (479,192.4) node [anchor=north west][inner sep=0.75pt]  [font=\tiny]  {$-1$};
		\draw (329,152.4) node [anchor=north west][inner sep=0.75pt]    {$\mathcal{D}_{1}$};
		\draw (423,150.4) node [anchor=north west][inner sep=0.75pt]    {$\mathcal{D}_{1}$};
		\draw (309,204.4) node [anchor=north west][inner sep=0.75pt]  [font=\tiny]  {$-\frac{3}{4}$};
		\draw (353,204.4) node [anchor=north west][inner sep=0.75pt]  [font=\tiny]  {$-\frac{1}{3}$};
		\draw (453,204.4) node [anchor=north west][inner sep=0.75pt]  [font=\tiny]  {$\frac{3}{4}$};
		\draw (409,204.4) node [anchor=north west][inner sep=0.75pt]  [font=\tiny]  {$\frac{1}{3}$};
		\draw (373,70.4) node [anchor=north west][inner sep=0.75pt]    {$\mathcal{D}_{2}$};
		\draw (279,173.4) node [anchor=north west][inner sep=0.75pt]  [font=\tiny]  {$-\frac{3}{4}$};
		\draw (279,130.4) node [anchor=north west][inner sep=0.75pt]  [font=\tiny]  {$-\frac{1}{3}$};
	\end{tikzpicture}
	\caption{Example~\ref{ex:rand_jump}: Half partition (left) and rectangle partition (right).} 
	\label{Ex4_domain}
\end{figure}

\begin{figure}[htp!]
	\centering
	\includegraphics[width=1\textwidth]{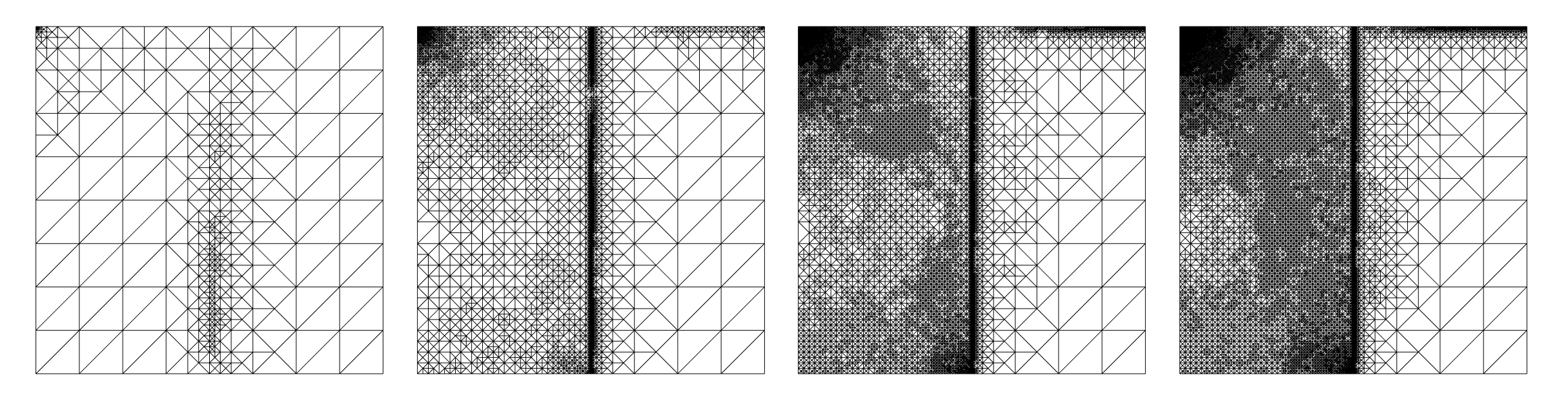} 
	\includegraphics[width=1\textwidth]{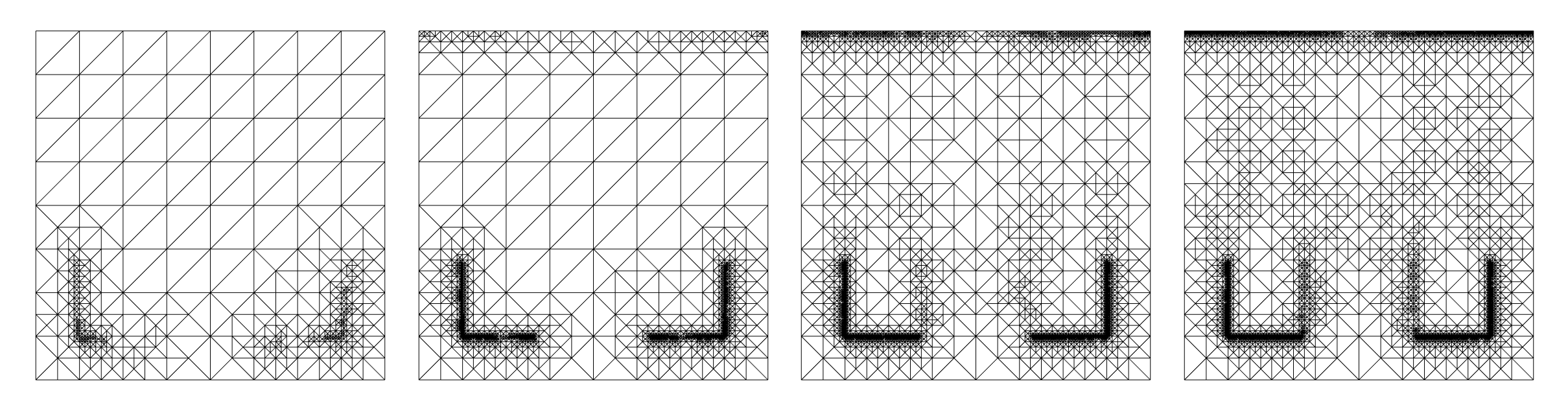} 
	\caption{Example~\ref{ex:rand_jump}: Process of adaptively refined triangulations obtained by Algorithm~\ref{alg:adap} with the marking parameters $\theta_h=0.5$, $\theta_q =0.5$ 
		for half partition (top) and rectangle partition (bottom).} \label{Ex4_mesh}	
\end{figure}

\begin{figure}[htp!]
	\centering
	\includegraphics[width=0.90\textwidth]{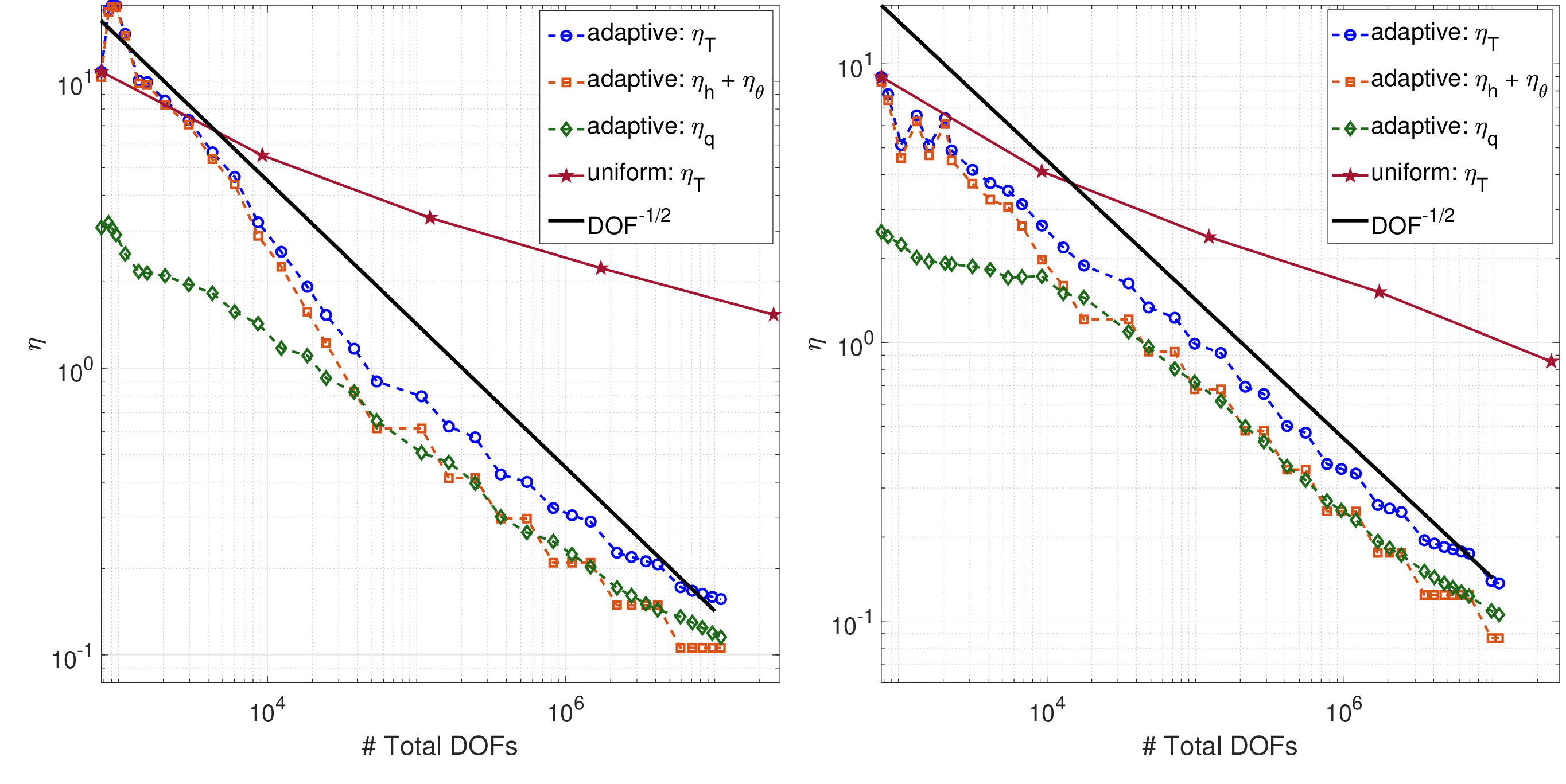} 
	\caption{Example~\ref{ex:rand_jump}: Behaviour of error estimators  on the adaptively and uniformly generated spatial/parametric spaces with the marking parameters $\theta_h=0.5$, 
		$\theta_q =0.5$ for half partition (left) and rectangle partition (right).} \label{Ex4_estimator}	
\end{figure} 
Figure~\ref{Ex4_mesh} displays the process of adaptively refined triangulations obtained by Algorithm~\ref{alg:adap} with the marking parameters $\theta_h=0.5$, $\theta_q =0.5$ 
for both case. Our estimator detects the changes in the mean of random variable well. Behavior of error estimator is also given in Figure~\ref{Ex4_estimator} by comparing the results 
on the adaptively and uniformly refined/enriched spaces. As expected, the results on adaptive ones yield better performance.

\section{Conclusion}\label{sec:conc}
In this paper, we have proposed an adaptive approach to approximate the statistical moments of convection dominated PDEs with random coefficients. 
We  have derived a residual--based estimator, consisting of spatial, data, 
and parametric contributions. Adding the truncation error coming from KL expansion, we have obtained a balance between the errors emerged from spatial and (stochastic) parametric domains. 
Different types of benchmark examples are tested and numerical simulations show that the proposed error estimator detects the localized regions well, where unwanted oscillations are displayed, and 
the adaptive approach displays 
a better behaviour than the uniform approach. Although the effectivity index tends to a constant as the iteration progress, it is not around the unity and therefore needs to 
be explored in the upcoming works. As a future study, the findings can be also extended to more complex  problems such as nonlinear PDEs and PDE-constrained optimization 
problems, and to different types of randomness such as boundary conditions and geometry.

\section*{Declaration of competing interest}
The authors declare that they have no known competing financial interests or personal relationships that could have appeared to influence the work reported in this paper.

\section*{Data availability}
No data was used for the research described in the article.

\section*{Acknowledgements}
This work has been supported by Scientific Research Projects Coordination Unit of Middle East Technical University under grant number GAP-705-2022-10821. HY acknowledges support from the BAGEP 2022 Award of the Science Academy. The authors also would like to express their sincere thanks to the anonymous referees for their most valuable suggestions.



\end{document}